\documentclass{article}

\usepackage{authblk}
\usepackage{amsmath, amssymb, amsthm}
\usepackage{graphicx,url,enumerate}
\usepackage[lined,boxed,commentsnumbered]{algorithm2e}

\newtheorem{theorem}{Theorem}[section]
\newtheorem{lemma}[theorem]{Lemma}
\newtheorem{proposition}[theorem]{Proposition}
\newtheorem{corollary}[theorem]{Corollary}
\theoremstyle{definition}
\newtheorem{definition}[theorem]{Definition}
\newtheorem{example}[theorem]{Example}
\newtheorem{remark}[theorem]{Remark}

\newcommand{\zexp}{\mathfrak{exp}\mspace{1mu}}

\newcommand{\zco}{\mathfrak{C}\mspace{1mu}}
\newcommand{\zre}{\mathfrak{R}\mspace{1mu}}

\newcommand{\zdu}{\mathfrak{D}\mspace{1mu}}
\newcommand{\CZ}{\mathbb{C}\mathfrak{Z}}
\newcommand{\Mat}{{\rm Mat}}
\newcommand{\RZ}{\mathbb{R}\mathfrak{Z}}
\newcommand{\Z}{\mathfrak{Z}}
\newcommand{\Ann}{{\rm Ann}_\Z}

\DeclareMathOperator{\sgn}{sgn}
\DeclareMathOperator{\rnk}{rank}
\DeclareMathOperator{\tr}{tr}

\DeclareMathOperator{\spn}{span}

\newcommand{\zm}{\mathbf{0}}
\newcommand{\mI}{\mathbb{I}}

\numberwithin{equation}{section}

\title{A Spectral Theorem for Zeon Matrices}
\author{G.~Stacey Staples\footnote{Email: sstaple@siue.edu}}
\affil{Department of Mathematics \&
Statistics\\
Southern Illinois University Edwardsville\\
Edwardsville, IL 62026-1653\\
USA}

\date{}  

\begin{document}

\maketitle

\begin{abstract}
In this paper, spectral properties of matrices with (complex) zeon entries are investigated.  
It is shown that when $A$ is an $m\times m$ self-adjoint matrix whose characteristic polynomial $\chi_A(u)$ has $m$ ``spectrally simple'' zeros $\lambda_1, \ldots, \lambda_m$ in the zeon algebra ${\mathbb{C}\mathfrak{Z}}$, there exist $m$  linearly independent normalized zeon eigenvectors $v_1, \ldots, v_m$ such that $A=\bigoplus_{j=1}^m \lambda_j\pi_j$, where $\pi_j=v_j{v_j}^\dag$ is a rank-one projection onto the zeon submodule ${\rm span}\{v_j\}$ for $j=1, \ldots, m$.   
\\
keywords: Spectral theory, Zeons
\end{abstract}

\section{Introduction}

Letting $Z_n$ denote the multiplicative semigroup generated by a collection of commuting null-square variables $\{\zeta_{\{i\}}:1\le i\le n\}$ and identity $1=\zeta_\varnothing$, the resulting $\mathbb{R}$-algebra,  denoted here by $\RZ_n$ has appeared in multiple guises over many years.  In recent years, $\RZ_n$ has come to be known as the {\em $n$-particle zeon\footnote{The name ``zeon algebra'' was coined by Feinsilver~\cite{Feinsilverzeons}, stressing their relationship to both bosons (commuting generators) and fermions (null-square generators).} algebra}.   They naturally arise as commutative subalgebras of fermions.

Combinatorial properties of zeons have been applied to graph enumeration problems, partition-dependent stochastic integrals, and routing problems in communication networks, as summarized in \cite{OCGraphs}.   More recent combinatorial applications include graph colorings~\cite{StaplesStellhorn2017} and Boolean satisfiability~\cite{DavisStaples2019}.   Combinatorial identities involving zeons have also been developed in papers by Neto~\cite{NetoJIS2014,NetoJIS2015,NetoJIS2016,NetoSIAM}.   
   
A permanent trace formula analogous to MacMahon's Master Theorem was presented and applied by Feinsilver and McSorley in \cite{FeinsilverMcSorleyIJC}, where the connections of zeons with permutation groups acting on sets and the Johnson association scheme were illustrated.  

Polynomials over the finite-dimensional complex zeon algebra $\CZ_n$ were first considered in \cite{simplezeonzeros}.   The current paper makes use of essential results to find eigenvalues of zeon matrices while also replacing finite-dimensional zeon algebras with the infinite-dimensional complex zeon algebra $\CZ$.  

Terminology and essential properties of the complex zeon algebra $\CZ$ are found in Section \ref{preliminaries}.   A discussion of $\CZ$-linear algebra, including nuances of linear independence and the $\CZ$-module $\CZ^m$, follows in Section \ref{essential zeon linear algebra}

Zeon matrices are viewed as $\CZ$-linear transformations of the module ${\CZ}^m$ in Section \ref{matrices as operators}, where their properties are investigated.  Our attention turns to zeon eigenvalues and eigenvectors in Section \ref{zeon eigenvalues}.  The zeon spectral theorem is established in Section \ref{zeon spectral theorem}.  

Matrix exponentials are considered in Section \ref{zeon matrix exponential}.  In particular, a number of cases are detailed in which the  series expansion $\exp(A)=\displaystyle\sum_{k=0}^\infty \frac{A^k}{k!}$ reduces to a finite sum.  The paper concludes with a summary and discussion of avenues for future work in Section \ref{conclusion}.

\section{Preliminaries}\label{preliminaries}

Let $\Z$ denote the Abelian semigroup generated by the collection $\{\zeta_{\{\ell\}}: \ell\in\mathbb{N}\}$, along with the identity 1, subject to the following multiplication rules:
\begin{gather*}
\zeta_{\{\ell\}}\,\zeta_{\{j\}}=\zeta_{\{\ell,j\}}=\zeta_{\{j\}}\,\zeta_{\{\ell\}}\,\,\text{ \rm
for }\ell\ne j,\text{ \rm and}\\
{\zeta_{\{\ell\}}}^2=0\,\,\text{ \rm for }\ell\ne 0.
\end{gather*}
The identity of $\Z$ is denoted by $1_\Z$.  Identifying $1_\Z$ with the scalar identity $1\in\mathbb{C}$, the associated semigroup algebra is denoted by $\CZ$.  In particular, we obtain the complex {\em zeon algebra} defined as follows.

\begin{definition}
Let $\CZ$ denote the infinite-dimensional complex Abelian algebra generated by the collection $\{\zeta_{\{i\}}: i\in \mathbb{N}\}$ along with the scalar $1=\zeta_\varnothing$ subject to the following multiplication rules:
\begin{gather*}
\zeta_{\{i\}}\,\zeta_{\{j\}}=\zeta_{\{i,j\}}=\zeta_{\{j\}}\,\zeta_{\{i\}}\,\,\text{ \rm
for }i\ne j,\text{ \rm and}\\
{\zeta_{\{i\}}}^2=0\,\,\text{ \rm for }i\in \mathbb{N}.
\end{gather*}
For each finite subset $I$ of $\mathbb{N}$, define $\displaystyle\zeta_{I}=\prod_{\iota\in I}\zeta_\iota$.  Letting the finite subsets of positive integers be denoted by $[\mathbb{N}]^{<\omega}$, the algebra $\CZ$ has a canonical basis of the form $\{\zeta_I: I\in [\mathbb{N}]^{<\omega}\}$.   Elements of this basis are referred to as the {\em basis blades} of $\CZ$.   The algebra $\CZ$ is called the  {\em (complex) zeon algebra}.
\end{definition}

As a vector space, this algebra has a canonical basis of  {\em basis blades} of the form $\{\zeta_I: I\subseteq [n]\}$.  The null-square property of the generators $\{\zeta_j:1\le j\le n\}$ guarantees that the product of two basis blades satisfies the following:\begin{equation}\label{blade product}
\zeta_I\zeta_J=\begin{cases}
\zeta_{I\cup J}& I\cap J=\varnothing,\\
0&\text{\rm otherwise.}
\end{cases}
\end{equation}

Any element $u\in\CZ$ can be expressed as a linear combination of basis blades indexed by finite subsets of $\mathbb{N}$.   It should be clear that $\CZ$ is graded.  For non-negative integer $k$, the {\em grade-$k$ part} of element $u=\sum_I u_I\zeta_I$ is defined as \begin{equation*}
\langle u\rangle_k=\sum_{\{I:|I|=k\}}u_I\zeta_I.
\end{equation*}

Given  $u=\sum_I u_I\zeta_I\in\CZ$, the {\em complex conjugate} of $u$ is defined by \begin{equation*}
\overline{u}=\sum_I \overline{u_I}\zeta_I.
\end{equation*}

The {\em real zeon algebra} $\Z$ is the subalgebra of $\CZ$ defined by \begin{equation*}
\Z=\{u\in\CZ: \overline{u}=u\}.
\end{equation*}

An inner product $\langle\cdot,\cdot\rangle:\CZ\to\mathbb{C}$ is defined by sesquilinear extension of \begin{equation*}
\langle u,\zeta_I\rangle = u_I,
\end{equation*}
where $u_I\in\mathbb{C}$ is the coefficient of $\zeta_I$ in the canonical expansion of $u$.
With this inner product in hand, the {\em (inner product) norm} of $u\in\CZ$ is defined by \begin{equation}\label{IP norm on CZ}
\Vert u\Vert=\left(\sum_I \vert u_I\vert^2\right)^{1/2}.
\end{equation}

\begin{definition}
For a zeon $u\ne 0$, it is useful to define the {\em minimal grade} of $u$ by \begin{equation}\label{mingrade def}
\natural u =\begin{cases}
\min\left\{k\in \mathbb{N}: \langle \zdu u\rangle_k\ne 0\right\}& \zdu u\ne 0,\\
0&u=\zco u.
\end{cases}
\end{equation}
Note that $\natural u=0$ if and only if $u$ is a scalar.  In this case, $u$ is said to be {\em trivial}.
\end{definition}

The maximal ideal consisting of nilpotent zeon elements will be denoted by \begin{equation*}
{\CZ}^\circ=\{u\in \CZ: \zco u=0\}.
\end{equation*}   
The multiplicative abelian group of invertible zeon elements is denoted by \begin{eqnarray*}
{\CZ}^\times&=&\CZ\setminus {\CZ}^\circ\\
&=&\{u\in \CZ: \zco u\ne 0\}.
\end{eqnarray*}

For convenience, arbitrary elements of $\CZ$ will be referred to simply as ``zeons.''  In what follows, it will be convenient to separate the scalar part of a zeon from the rest of it. To this end, for $z\in\CZ_n$ we write $\zco z=\langle z\rangle_0$, the {\em complex (scalar) part} of $z$, and $\mathfrak{D}z=z-\zco z$, the {\em dual part}~\footnote{The term ``dual'' here is motivated by regarding zeons as higher-dimensional dual numbers.} of $z$.

We note that  $u\in\CZ_n$ is invertible if and only if $\zco u\ne 0$.  Moreover, the multiplicative inverse of $u$ is unique.  

\begin{proposition}\label{zeon inverses}
Let $u\in\CZ$, and let $\kappa$ denote the index of nilpotency~\footnote{In particular, $\kappa$ is the least positive integer such that $(\mathfrak{D}u)^\kappa=0$.} of $\mathfrak{D}u$.    It follows that $u$ is uniquely invertible if and only if $\zco u\ne0$, and the inverse  is given by \begin{equation}\label{inv def}
u^{-1}=\frac{1}{\zco u}\sum_{j=0}^{\kappa-1}(-1)^{j}(\zco u)^{-j}(\mathfrak{D}u)^{j}.
\end{equation}
\end{proposition}

\begin{proof}
Proof is by direct computation, observing that a telescoping sum is obtained:
\begin{align*}
u\left(\frac{1}{\zco u}\sum_{j=0}^{\kappa-1}(-1)^{j}(\zco u)^{-j}(\zdu u)^j\right)&=
(\zco u+\zdu u)\sum_{j=0}^{\kappa-1}(-1)^{j}({\zco u})^{-(j+1)}(\zdu u)^j\\
&=\sum_{j=0}^{\kappa-1}(-1)^{j}({\zco u})^{-j}(\zdu u)^j\\
&\hskip10pt+\sum_{j=0}^{\kappa-1}(-1)^{j}({\zco u})^{-(j+1)}(\zdu u)^{j+1}\\
&=1\pm ({\zco u})^{-(\kappa)}(\zdu u)^\kappa=1.
\end{align*}
\end{proof}

\begin{remark}
Another way to see the result of Proposition \ref{zeon inverses} is to first recall that if the geometric series $\sum_{j=0}^\infty x^j$ converges, its limit is $\frac{1}{1-x}$.  Letting $a=\zco u\ne 0$ and writing $u=a+\zdu u$, we see that \begin{align*}
u^{-1}=(a+\zdu u)^{-1}&=a^{-1}\frac{1}{1- (-a\zdu u)}\\
&=a^{-1}\sum_{j=0}^{\kappa-1}(-1)^{j}a^{-j}(\zdu u)^j,
\end{align*}
where nilpotency of $\zdu u$ reduces the infinite series to a finite sum, eliminating any concern about lack of convergence.  
\end{remark}

Finally, we turn to zeon contraction operators.  These operators are useful for characterizing solutions of $\CZ$-linear equations.  

\begin{definition}
The {\em left- and right-contraction operators} $\lrcorner$ and $\llcorner$ are defined  by $\CZ$-linear extension of \begin{equation}\label{contraction def}
\zeta_I\lrcorner \zeta_J=\zeta_J\llcorner\zeta_I=\begin{cases}
\zeta_{J\setminus I}&\text{\rm if }I\subseteq J,\\
0&\text{\rm otherwise.}
\end{cases}
\end{equation}
Identifying $\zeta_\varnothing$ with the unit scalar $1$, we have $\zeta_I\lrcorner\zeta_I=1$.
\end{definition}

\begin{example}
Let $u=3\zeta_{\{1,2\}}-4\zeta_{\{2,4,5\}}$.  Then,
\begin{align*}
\zeta_{\{1\}}\lrcorner u=u\llcorner \zeta_{\{1\}}&=3\zeta_{\{2\}},&\zeta_{\{1,2\}}\lrcorner u=u\llcorner \zeta_{\{1,2\}}&=3,
\end{align*}
and
\begin{align*}
\zeta_{\{2\}}\lrcorner u=u\llcorner \zeta_{\{2\}}&=3\zeta_{\{1\}}-4\zeta_{\{4,5\}}.
\end{align*}

\end{example}

\subsubsection{Zeon Annihilators}

Given a nonzero nilpotent zeon element  $q=\sum_{I}q_I\zeta_I\in\CZ^\circ$ and nontrivial basis blade $\zeta_X$, it is evident that $\zeta_{X} q=0$ if and only if $X\cap I\ne\varnothing$ whenever $q_I\ne 0$.  That is, $\zeta_X\zeta_I=0$ whenever $q_I\ne 0$.  

\begin{definition}
Let $q=\sum_{I}q_I\zeta_I\in\CZ^\circ$.  A basis blade $\zeta_X$ is an {\em annihilator} of $q$ if $\zeta_X q=0$.  
\end{definition}

An easy way to obtain an annihilator of $q\in\CZ^\circ$ is to simply take the union of all multi-indices appearing in the canonical expansion of $q$.  

\begin{definition}\label{index support def}
Let $q=\sum_{I}q_I\zeta_I\in\CZ$.  The {\em index support of $q$} is defined to be \begin{equation}
[q]=\bigcup_{\{I: q_I\ne 0\}}I.
\end{equation}
\end{definition}

The index support of a nilpotent $q$ is used to obtain a null monomial that ``annihilates'' $q$; i.e., $q\zeta_{[q]}=0$.   For this reason, $\zeta_{[q]}$ will be referred to as a {\em minimal annihilator} of $q\in\CZ^\circ$.  More generally, $q\zeta_{[q]}=(\zco q) \zeta_{[q]}$ for arbitrary $q\in\CZ$, so that $\zeta_{[q]}$ is an annihilator of $\zdu q$. 

\begin{example}
Let $q=3+4\zeta_{\{2\}}-5\zeta_{\{1,3,4\}}.$  Then $[q]=\{1,2,3,4\}$ and 
\begin{align*}
q\zeta_{[q]}&=(3+4\zeta_{\{2\}}-5\zeta_{\{1,3,4\}})\zeta_{\{1,2,3,4\}}\\
&=3\zeta_{\{1,2,3,4\}}.
\end{align*}  
\end{example}

\begin{lemma}[Annihilator Ideals]\label{annihilator ideal}
Given nilpotent $q\in\CZ^\circ$, define $\Ann(q)=\{\alpha\in\CZ^\circ: \alpha q=0\}$.  Then $\Ann(q)$ is an ideal of $\CZ^\circ$.
\end{lemma}

\begin{proof}
Clearly, $\eta, \mu\in\Ann(q)$ implies $(\eta+\mu)q=\eta q+\mu q=0$ and $(\alpha \eta)q=\alpha (\eta q)=0$ so that $\Ann(q)$ is a subring of $\CZ^\circ$.  Moreover, for any $\alpha\in\CZ$, we have $(\alpha\eta)q=\alpha(\eta q)=\alpha\cdot0=0$.
\end{proof}

\begin{remark}
Given $\mathbf{v}=(v_1, v_2, \ldots, v_m)^\intercal\in(\CZ^\circ)^m$, we see that  $\omega \mathbf{v}=\mathbf{0}$ for any $\omega\in\Ann(\{v_1, \ldots, v_m\})$.
\end{remark}

\section{$\CZ$-Linear Algebra}\label{essential zeon linear algebra}

We begin by considering equations of the form $\alpha u=\beta$, where $\alpha, \beta\in\CZ$ and immediately observe that the equation has the unique solution $u=\alpha^{-1}\beta$ if $\alpha\in\CZ^\times$.  On the other hand, if $\alpha\in\CZ^\circ$ and $\beta\in\CZ^\times$, the equation has no solution because $\CZ^\circ$ is an ideal.  

When $\alpha,\beta\in\CZ^\circ$, the following cases must be considered.  

\begin{description}
\item [Case 1: $\beta=0$.]
If $\beta=0$, then any $u\in\Ann(\alpha)$ is a solution.   

\item [Case 2: $\natural \alpha>\natural\beta$.]
When $\natural \alpha>\natural \beta$, the equation has no solutions because we either have $\alpha u=0$ or $\natural (\alpha u)\ge \natural\alpha$ for any $u\in\CZ$.  

\item [Case 3: $\natural \alpha=\natural\beta$.]
If there exists a scalar $c\in\mathbb{C}$ such that $\langle\beta\rangle_k=c\langle\alpha\rangle_k$ for each $k=1, \ldots, \sharp\beta$, then the equation $\alpha u=\beta$ has solution set $c+\Ann \alpha$, since $\alpha (c+\eta)=c\alpha+0=\beta$ for any $\eta\in\Ann\alpha$.  If no such scalar exists, the equation has no solutions.

\item [Case 4: $\natural \alpha<\natural\beta$.]  The only possible solutions are elements of the form $u=c\beta\llcorner \alpha+\eta$, where $\eta\in\Ann(\alpha)$ and $c\in\mathbb{C}$.   However, solutions need not exist.
\end{description}

In all four cases, the equation $\alpha u=\beta$ has either no solution or infinitely many solutions.  The single homogeneous equation $\alpha u=0$ is thereby considered to be {\em linearly dependent} when $\alpha\in\CZ^\circ$.  A homogeneous linear equation in $m$ variables is considered to be {\em linearly independent} if and only if one or more coefficients is invertible; i.e., for some $\ell\in\{1, \ldots, m\}$, \begin{equation*}
\sum_{j=1}^m \alpha_j u_j=0\Rightarrow u_\ell=-{\alpha_\ell}^{-1}\sum_{j\ne \ell}\alpha_j u_j.
\end{equation*}

Summarizing, we have established the following theorem (noting that cases 3 and 4 of our previous discussion can be combined).

\begin{theorem}
Given $\alpha,\beta\in\CZ$, let $S=\{u\in\CZ: \alpha u-\beta=0\}$,  the solution set of the zeon linear equation $\alpha u-\beta=0$.  If $\alpha\in\CZ^\times$, then $S=\{\alpha^{-1}\beta\}$, the unique solution of the equation. Assuming $\alpha\in\CZ^\circ$  is nilpotent, the solutions are characterized as follows.  

\begin{enumerate}[i.]
\item $S=\varnothing$  if $\beta\in\CZ^\times$ or if $\natural\alpha>\natural\beta$;
\item $S= \Ann(\alpha)$ if $\beta=0$;
\item $S\subseteq\{c\beta\llcorner\alpha+\eta: \eta\in\Ann(\alpha)\}$ for some $c\in\mathbb{C}$ if $\natural\alpha\le\natural\beta$.   
\end{enumerate}  
\end{theorem}

\begin{example}
Let $\alpha=2\zeta_{\{1,2\}}-5\zeta_{\{2,3\}}$.  Then $\Ann(\alpha)=\langle \zeta_{\{2\}}\rangle$.  It follows that $\alpha u=0$ for any $u\in\Ann(\alpha)$.

Let $\beta=\zeta_{\{1\}}+\zeta_{\{2\}}$.  Since $\natural\alpha=2>1=\natural \beta$, the equation $\alpha u=\beta$ has no solution.

Let $\beta=4\zeta_{\{1,2\}}-10\zeta_{\{2,3\}}$.  Then, $\alpha u=\beta$ has solutions of the form $u=2+\eta$ for any $\eta\in\Ann(\alpha)$.  For example, choosing $u=2+\zeta_{\{2\}}\in 2+\Ann(\alpha)$ gives \begin{align*}
\alpha(2+\zeta_{\{2\}})&=(2\zeta_{\{1,2\}}-5\zeta_{\{2,3\}})(2+\zeta_{\{2\}})\\
&=4\zeta_{\{1,2\}}-10\zeta_{\{2,3\}}=\beta.
\end{align*}
Finally we consider the case $\natural\alpha<\natural\beta$, letting $\beta=\zeta_{\{1,2,4\}}-\zeta_{\{2,3,4\}}$.  Here, \begin{equation*}
\beta\llcorner\alpha=(\zeta_{\{1,2,4\}}-\zeta_{\{2,3,4\}})\llcorner(2\zeta_{\{1,2\}}-5\zeta_{\{2,3\}})=2\zeta_{\{4\}}-5\zeta_{\{4\}}=-3\zeta_{\{4\}}.
\end{equation*}
So, letting $u=\zeta_{\{4\}}+\eta$, where $\eta\in\Ann(\alpha)$, we obtain\begin{align*}
\alpha u = (2\zeta_{\{1,2\}}-5\zeta_{\{2,3\}})(\zeta_{\{4\}}+\eta)&=
2\zeta_{\{1,2,4\}}-5\zeta_{\{2,3,4\}}+0=\beta.
\end{align*}
\end{example}

\subsubsection*{Multivariable $\CZ$-Linear Equations}

Consider now the multivariable $\CZ$-linear equation \begin{equation*}
\alpha_1 u_1+\alpha_2 u_2+\cdots+ \alpha_n u_n=\gamma.
\end{equation*}
If any of the coefficients $\{\alpha_j: 1\le j\le n\}$ is invertible, the equation has infinitely many solutions.  In particular, if $\alpha_j\in\CZ^\times$, then the equation is solved for $u_j$ by \begin{equation*}
u_j={\alpha_j}^{-1}\left(\gamma-\sum_{\ell\ne j}\alpha_\ell u_\ell\right),
\end{equation*}
where the remaining variables $u_\ell$ are free for $\ell\ne j$.  

When all coefficients are nilpotent the situation is more complicated, and solutions need not exist. However, solutions always exist for the case $\gamma=0$, since the ideal $\Ann(\alpha_1)\times\cdots\times \Ann(\alpha_n)$ is a subset of the solution space of $\alpha_1 u_1+\cdots+\alpha_n u_n=0$ when $\alpha_j\in\CZ^\circ$ for $j=1, \ldots, n$.

\subsection{Linear Independence in $\CZ^m$}

Given positive integer $m$, the collection of $m$-tuples of zeon elements constitutes the $\CZ$-module
\begin{align*}
{\CZ}^m&=\underset{m\, \text{\rm times}}{\underbrace{\CZ\times \cdots \times \CZ}}\\
&=\{(\alpha_1, \alpha_2, \ldots, \alpha_m): \alpha_1, \ldots, \alpha_m\in \CZ\}.  
\end{align*}

A system of $m$ zeon linear equations in $n$ variables $\{u_1, \ldots, u_n\}$ is naturally represented by a matrix equation $Au=b$ where $A\in\Mat(m\times n;\CZ)$, $b\in\CZ^m$ and $u=(u_1, \ldots, u_n)^\intercal$.

A collection $\{w_1, \ldots, w_n\}\subset\CZ^m$ is said to be {\em $\CZ$-linearly independent} if for $\alpha_1, \ldots, \alpha_n\in\CZ$, the following implication holds:
\begin{equation*}
\alpha_1 w_1+\cdots + \alpha_n w_n=0 \Rightarrow \alpha_1=\cdots=\alpha_n =0.
\end{equation*}

Note that if $w\in (\CZ^\circ)^m$, $\alpha w=0\Rightarrow \alpha\in\Ann w\ne \{0\}$.  Hence, a singleton from the module $\CZ^m$ is only $\CZ$-linearly independent if it has one or more invertible entries.  

The {\em rank} of a zeon matrix is the number of $\CZ$-linearly independent rows (or columns) in the matrix.  Hence, $\rnk A=\rnk\zco A$.  

Viewing $A\in\Mat(m\times n;\CZ)$ as a $\CZ$-linear transformation $A:\CZ^n\to\CZ^m$, the rank-nullity theorem guarantees 
\begin{equation}
m=\rnk(A)+\ker(A).
\end{equation}

Given a linear equation $\alpha u=\beta$ having a unique solution (i.e., $\alpha\in\CZ^\times$),  multiplying both sides by a nilpotent zeon element $q$ may result in an equation  having infinitely many solutions.  For example, \begin{equation*}
q\alpha u=q\beta\Rightarrow q(\alpha u-\beta)=0
\end{equation*}
has solutions $u=\alpha^{-1}\beta$ (the unique solution to the original equation), and $u\in\CZ$ such that $\alpha u-\beta\in\Ann q$.
On the other hand, multiplying the equation $\alpha u=\beta$ by an invertible zeon $q$ leaves the solution invariant, since $q(\alpha u-\beta)=0$ if and only if $\alpha u-\beta=0$ when $q\in\CZ^\times$.

Because of this, we will restrict our elementary row operations to consider only {\em invertible} zeon multiples of rows.  

\begin{example}
Suppose $\alpha_1\in\CZ^\times$ and $\alpha_2\in\CZ^\circ$.  Then \begin{equation*}
\alpha_1 u_1+\alpha_2u_2=0\Rightarrow u_1=-{\alpha_1}^{-1}\alpha_2 u_2\in\CZ^\circ.
\end{equation*}
The complete set of solutions is thus $S=\{(-{\alpha_1}^{-1}\alpha_2 \gamma,\gamma):\gamma\in\CZ^\circ\}$.
Solutions include $\{(0,\gamma): \gamma\in\Ann\alpha_2\}$.
\end{example}

\section{Zeon Matrices as $\CZ$-Linear Operators on $\CZ^m$}\label{matrices as operators}

Given positive integer $m$, the algebra of square zeon matrices $\Mat(m;\CZ)$ will be regarded as the algebra of $\CZ$-linear operators on $\CZ^m$.

For $m\in \mathbb{N}$, let $\Mat(m;\CZ)$ denote the algebra of $m\times m$ matrices having entries from the zeon algebra $\CZ$.  A matrix $A\in\Mat(m;\CZ)$ can naturally be written as a sum of the form $A=A_\varnothing+\mathfrak{A}$ of a complex-valued matrix $A_\varnothing = \zco A$ and a nilpotent, zeon-valued matrix $\mathfrak{A}=\zdu A$.   
Note that the invertible elements of $\Mat(m; \CZ)$ constitute a multiplicative group.  

Given $A\in\Mat(m;\CZ)$, it is not difficult to verify that  \begin{equation*}
A(\alpha\mathbf{x}_1+\mathbf{x}_2)=\alpha A\mathbf{x}_1+A\mathbf{x}_2
\end{equation*}
 for $\mathbf{x}_1, \mathbf{x}_2\in{\CZ}^m$ and $\alpha\in \CZ$; i.e., $A$ is a $\CZ$-linear operator on ${\CZ}^m$,

\subsection*{The Determinant}

The determinant of $A\in\Mat(m;\CZ)$ is defined in the usual way by \begin{equation*}
|A|=\sum_{\sigma\in \mathcal{S}_m}\sgn(\sigma)\prod_{j=1}^m a_{j\,\sigma(j)},
\end{equation*}
where $\mathcal{S}_m$ is the symmetric group of order $m!$ and $\sgn (\sigma)$ is the signature of the permutation $\sigma$.  

\begin{lemma}\label{scalar det}
For all $A\in\Mat(m;\CZ)$,  $\zco |A| = |\zco A|$. 
\end{lemma}

\begin{proof}
Given arbitrary $u,v\in \CZ$,  one observes that \begin{align*}
\zco (uv)&=\zco \left((\zco u+\zdu u)(\zco v+\zdu v)\right)\\
&=\zco \left(\zco u \zco v+\zco u\zdu v+\zco v\zdu u+\zdu u\zdu v\right)\\
&=\zco u \zco v.
\end{align*}
Extending to arbitrary products in the determinant, one obtains
\begin{align*}
\zco |A|
&=\zco \left(\sum_{\sigma\in \mathcal{S}_m}\sgn(\sigma)\prod_{j=1}^m (\zco a_{j\,\sigma(j)}+\zdu a_{j\,\sigma(j)})\right)\\
&=\sum_{\sigma\in \mathcal{S}_m}\sgn(\sigma)\zco \left(\prod_{j=1}^m (\zco a_{j\,\sigma(j)}+\zdu a_{j\,\sigma(j)})\right)\\
&=\sum_{\sigma\in \mathcal{S}_m}\sgn(\sigma)\prod_{j=1}^m \zco a_{j\,\sigma(j)}.
\end{align*}
\end{proof}

 \begin{definition}
A matrix $A\in\Mat(m;\CZ)$ is said to be {\em singular} if it is not invertible. In particular, $A$ is singular if and only if $|A|\in\CZ^\circ$.   
\end{definition}

\begin{proposition}\label{complex endo}
The mapping $A\mapsto \zco A$ is an algebra homomorphism  \break $\Mat(m;\CZ)\to {\rm End}(\mathbb{C}^m)$.
\end{proposition}

\begin{proof}
Recognizing that ${\rm End}(\mathbb{C}^m)$ is isomorphic to the algebra of $m\times m$ complex matrices, the result is established by verifying that $\zco(A+B)=\zco A+\zco B$ and that $\zco(AB)=(\zco A)(\zco B)$ for matrices $A,B\in \Mat(m;\CZ)$. \end{proof}

By Proposition \ref{complex endo}, $\zco(A^k)=(\zco A)^k$.   Hence, the following corollary.

\begin{corollary}
A matrix $A\in\Mat(m;\CZ)$ is singular if and only if $\zco A$ is singular.    
\end{corollary}

The following lemma is a natural consequence of the determinant definition.

\begin{lemma}[Properties of the determinant]
Let $A$ and $B$ be $m\times m$ matrices over $\CZ$, and let $\alpha\in\CZ$.  Then the following hold:\begin{align*}
|AB|&=|A| |B|,\\
|\alpha A|&=\alpha^m |A|.
\end{align*}
In particular, $|A^{-1}|=|A|^{-1}$ when $A$ is invertible.
\end{lemma}

The following result was initially established for matrices over real zeon algebras.   Its proof is straightforward, but the interested reader can find details in \cite{zeon laplacian}. 

\begin{lemma}
Let $A\in\Mat(m;\CZ)$.  Then, $A$ is nilpotent if and only if $\zco A$ is nilpotent.
\end{lemma}

\subsection*{The Matrix Inverse}\label{matrix inverse}

The next proposition was established for real zeon matrices in \cite{zeon laplacian}.  The proof for the complex generalization differs only by substituting $\zco A$ for $\zre A$.

\begin{description}
\item [Zeon Matrix Inverse]
Let $A=(a_{ij})$ be a square matrix having entries from $\CZ$, and write $A=\zco A + \mathfrak{D}A$, where $\zco A=(\zco a_{ij})$.  It follows that $A$ is invertible if and only if $\zco A$ is invertible.  In this case, the inverse is given by \begin{equation*}
A^{-1}=(\zco A)^{-1}\sum_{\ell=0}^{\kappa(\mathfrak{D}A(\zco A)^{-1})-1}(-1)^\ell(\mathfrak{D}A(\zco A)^{-1})^{\ell}.
\end{equation*}
\end{description}

In light of Lemma \ref{scalar det}, the next corollary follows immediately.
\begin{corollary}
A matrix $A\in\Mat(m;\CZ)$ is singular if and only if $|A|$ is nilpotent.
\end{corollary}

\subsubsection*{Zeon Gaussian Elimination}

It is easily verified that the following elementary (row) operations have the familiar effects on the matrix determinant.   If $A\mapsto A'$ via an elementary operation, then the effects are as follows:
\begin{enumerate}[i.]
\item exchanging two rows changes the sign, i.e., $|A'|=-|A|$;
\item multiplying a row by an invertible\footnote{Considering linear equations represented by the rows of a zeon matrix, we note that multiplying by a nilpotent zeon constant can change the solution set.  Hence, only multiplication by an invertible zeon element will be considered an elementary row operation.} zeon constant $u$ implies $|A'|=u|A|$;
\item adding any zeon multiple of one row to another gives $|A|=|A'|$.
\end{enumerate}

It is important to note that pivot elements of a zeon matrix must be invertible.  Hence, not every zeon matrix can be placed in row echelon form in the usual sense.   Without altering solution sets of matrix equations, a zeon matrix can be reduced (by elementary row operations) into a reduced form satisfying the following:
\begin{enumerate}
\item the left-most invertible entry (i.e., the pivot) of each row is to the right of the left-most invertible entry of every row above;
\item entries below each pivot are all zero.
\end{enumerate}

As in the scalar matrix case, elementary zeon matrices are defined as matrices obtained from the identity matrix by elementary zeon row operations.   All elementary matrices are invertible. 

\begin{example}
Consider the matrix $A$ given by 
\begin{equation*}
A=\left(
\begin{array}{ccc}
 2+\zeta _{\{1\}}& \zeta _{\{2\}} & 0 \\
 \zeta _{\{2\}} & 2-\zeta _{\{2\}} & 3 \zeta _{\{1,2,3\}} \\
 0 & 3 \zeta _{\{1,2,3\}} & 1-\zeta _{\{1,2,3\}}+\zeta _{\{1\}} \\
\end{array}
\right),
\end{equation*}
along with elementary matrix $E_1$ exchanging rows 1 and 2, and the elementary matrix $E_2$ that multiplies row 3 by the constant $2+3\zeta_{\{1,2\}}$, as given here
\begin{align*}
E_1&=\left(
\begin{array}{ccc}
 0 & 1 & 0 \\
 1 & 0 & 0 \\
 0 & 0 & 1 \\
\end{array}
\right);\hskip10pt 
E_2=\left(
\begin{array}{ccc}
 1 & 0 & 0 \\
 0 & 1 & 0 \\
 0 & 0 & 2+3 \zeta _{\{1,2\}} \\
\end{array}
\right).
\end{align*}
Then,
\begin{equation*}
E_1 A=\left(
\begin{array}{ccc}
 \zeta _{\{2\}} & 2-\zeta _{\{2\}} & 3 \zeta _{\{1,2,3\}} \\
 2+\zeta _{\{1\}} & \zeta _{\{2\}} & 0 \\
 0 & 3 \zeta _{\{1,2,3\}} & 1-\zeta _{\{1,2,3\}}+\zeta _{\{1\}} \\
\end{array}
\right)
\end{equation*}
and 
\begin{equation*}
E_2 A=\left(
\begin{array}{ccc}
 2+\zeta _{\{1\}} & \zeta _{\{2\}} & 0 \\
 \zeta _{\{2\}} & 2-\zeta _{\{2\}} & 3 \zeta _{\{1,2,3\}} \\
 0 & 6 \zeta _{\{1,2,3\}} & 2+3 \zeta _{\{1,2\}}-2 \zeta _{\{1,2,3\}}+2 \zeta _{\{1\}} \\
\end{array}
\right).
\end{equation*}

Further,  \begin{align*}
|A|&=4-3 \zeta _{\{1,2\}}-4 \zeta _{\{1,2,3\}}+6 \zeta _{\{1\}}-2 \zeta _{\{2\}},\\
|E_1 A|&=-4+3 \zeta _{\{1,2\}}+4 \zeta _{\{1,2,3\}}-6 \zeta _{\{1\}}+2 \zeta _{\{2\}}=-|A|,\\
|E_2 A|&=8+6 \zeta _{\{1,2\}}-8 \zeta _{\{1,2,3\}}+12 \zeta _{\{1\}}-4 \zeta _{\{2\}}=(2+3 \zeta _{\{1,2\}})|A|.
\end{align*}
\end{example}

\subsection{The Kernel of a Zeon Matrix}

Given $A\in\Mat(m;\CZ)$, the {\em kernel} of $A$ is the $\CZ$-module
$\ker(A)=\{u\in\CZ^m: Au=0\}$.  Any invertible zeon matrix has trivial kernel, since $Ax=0\Leftrightarrow x=A^{-1}0=0$.  Singular zeon matrices have nilpotent determinant.  

On the other hand, when $X\in\Mat(m;\CZ^\circ)$, i.e. $X$ is a rank-zero matrix with no invertible entries, we begin by considering the linear zeon equation \begin{equation*}
\alpha_1u_1+\alpha_2 u_2+\cdots+\alpha_m u_m=0,
\end{equation*}
where $\alpha_\ell\in\CZ^\circ$ for $\ell=1, \ldots, m$.  An obvious subset of solutions is the Cartesian product of the coefficients' annihilators:  \begin{align*}
(u_1, \ldots, u_m)&\in \Ann \alpha_1\times\Ann \alpha_2\times\cdots\Ann \alpha_m\\
&\Rightarrow \alpha_1u_1+\alpha_2 u_2+\cdots+\alpha_m u_m=0.
\end{align*}

Regarding the coefficients $\alpha_j$ as entries of the $i$th row of $X$, it is convenient to write $X=(\xi_1\vert \xi_2\vert\cdots\vert \xi_m)$ where $\xi_\ell\in(\CZ^\circ)^m$ is the $\ell$th column of $X$ and define the annihilator of $\xi_\ell$ in the natural way:  \begin{equation*}
\Ann \xi_\ell=\{u\in\CZ^\circ: u\,\xi_\ell=\zm\}.
\end{equation*}  
Hence, $ \Ann \xi_1\times\Ann \xi_2\times\cdots\Ann \xi_m\subseteq \ker(X)$ when $X\in\Mat(m;\CZ^\circ)$.

\begin{remark}
In fact, when $X\in\Mat(m;\CZ^\circ)$, there exists a basis blade $\zeta_{[X]}$ satisfying $\zeta_{[X]}X=\zm$.  Thus, \begin{equation*}
\langle\zeta^{[X]}\rangle^m\subseteq \Ann \xi_1\times\Ann \xi_2\times\cdots\Ann \xi_m\subseteq \ker(X),
\end{equation*}
where $\langle\zeta^{[X]}\rangle$ denotes the principal ideal of $\CZ$ generated by $\zeta_{[X]}$.
\end{remark} 

All singular zeon matrices have nontrivial kernel.  If $X\in\Mat(m;\CZ)$ is of rank $k<m$, then $X$ must have $m-k$ columns that will be nilpotent after Gaussian elimination.   For purposes of this paper, matrices of rank $m-1$ in $\Mat(m;\CZ)$ are of particular interest, as these will arise in the pursuit of eigenvectors.  

\section{Eigenvalues, Eigenvectors, and the Characteristic Polynomial}\label{zeon eigenvalues}

The interested reader is directed to the paper~\cite{simplezeonzeros} for a detailed study of zeon polynomials and their zeros.  For purposes of computing eigenvalues and eigenvectors of zeon matrices, only essential results are recalled here.

Given a complex zeon polynomial $\varphi(u)=\alpha_m u^m+\cdots+\alpha_1 u+\alpha_0$, a {\em complex polynomial} $f_\varphi:\mathbb{C}\to\mathbb{C}$ is induced by \begin{equation*}
f_\varphi(z)=\sum_{\ell=0}^m (\zco \alpha_\ell) z^\ell.
\end{equation*}
It follows that \begin{equation*}
f_\varphi(\zco u)=\sum_{\ell=0}^m (\zco \alpha_\ell) (\zco u)^\ell=\zco(\varphi(u)),\end{equation*}
so that $f_\varphi\circ\zco=\zco\circ \varphi$.

\subsection{Spectrally Simple Zeros of Complex Zeon Polynomials}

When the induced polynomial $f_\varphi(z)$ has a multiple root $w_0\in\mathbb{C}$, $\varphi(u)$ may or may not have a zero $w$ satisfying $\zco w=w_0$.  If it does, there are infinitely many!  For this reason, we restrict our attention to zeon polynomials whose induced complex polynomials have only simple zeros.

Letting $\varphi(u)$ be a nonconstant monic zeon polynomial, we consider $\lambda\in{\CZ}$ to be a {\em simple} zero of $\varphi$ if $\varphi(u)=(u-\lambda)g(u)$ for some zeon polynomial $g$ satisfying $g(\lambda)\ne 0$.  Recalling that the spectrum of an element $u$ in a unital algebra is the collection of scalars $\lambda$ for which $u-\lambda$ is not invertible, it follows that when $u\in\CZ$, the spectrum of $u$ is the singleton $\{\lambda=\zco u\}$.   

\begin{definition}
A zero $\lambda_0\in{\CZ}$ of $\varphi(u)$ is said to be a {\em spectrally simple} if $\zco \lambda_0$ is a simple zero of the complex polynomial $f_\varphi(z)$.
\end{definition}  

As seen in \cite{simplezeonzeros}, if $f_\varphi(z)$ has simple complex zero $z_0$, then $\varphi$ has a unique zeon zero $\lambda$ satisfying $\zco \lambda = u_0$.

Given $\varphi(u)\in\CZ[u]$ of degree $m\ge 1$, if the induced polynomial $f_\varphi(z)$ is a nonconstant complex polynomial whose zeros are all simple, then $\varphi(u)$ has exactly $m$ spectrally simple complex zeon zeros.  In this case, we say $\varphi$ {\em splits} over $\CZ$.

\subsubsection{Zeon Eigenvalues}

Given a square matrix $A\in \Mat(m;\CZ)$,  the {\em characteristic polynomial} of $A$ is defined in the usual way as the zeon polynomial $\chi_A(t)=|t\mI-A|$.   

\begin{definition}
Let $A\in \Mat(m;\CZ)$.  The {\em zeon eigenvalues} of $A$ are defined to be the spectrally simple zeros of the characteristic polynomial $\chi_A(t)$.  
\end{definition}

The following theorem formalizes the relationship between eigenvalues of the scalar matrix $\zco A$ and eigenvalues of $A\in\Mat(m;\CZ)$.  Its proof follows from our previous discussion.

\begin{theorem}
Let $A\in \Mat(m;\CZ)$.  If $\chi_{\zco A}(u)$ has simple complex zero $u_0$, then $A$ has a unique zeon eigenvalue $\lambda$ satisfying $\zco \lambda = u_0$.  
\end{theorem}

It is easy to verify that if $\lambda$ is a zeon eigenvalue of $A$, then $\zco \lambda$ is a complex eigenvalue of $\zco A$.  The next lemma is an immediate consequence.

\begin{lemma}
If $A$ is invertible, then all eigenvalues of $A$ are invertible.
\end{lemma}

\begin{proof}
If $A$ is invertible, then all eigenvalues of $\zco A$ are nonzero.  Hence, the scalar part of any zero of $\chi_A(u)$ must be nonzero.
\end{proof}

\subsection{Zeon eigenvectors}

Recall that for any nonzero $v\in {\CZ}^m$, the singleton $\{v\}$ can only be linearly independent if it has at least one invertible component.  That is, $\{v\}$ linearly independent implies $v\notin ({\CZ}^\circ)^m$.  

Given a spectrally simple zeon matrix $A\in\Mat(m; \CZ)$ having zeon eigenvalue $\lambda$, a zeon eigenvector $\xi$ associated with $\lambda$ will be a $\CZ$-linearly independent element of $\CZ^m$.  In particular, $\xi\in\ker(\lambda\mI-A)$.  The following lemma provides an essential tool for computing zeon eigenvectors.

\begin{lemma}\label{row reduction I}
Let $A\in\Mat(m;\CZ)$ be a zeon matrix of rank $m-1$.  If the $k$th column of $A$ consists entirely of nilpotent zeon elements, then there exists an invertible zeon matrix $Q$ such that $|QA|=\pm |A|$ and $QA$ is of the form \begin{equation}\label{QA matrix}
QA=\begin{pmatrix}
\alpha_1&0 &\cdots&0&\eta_1&0&0&\cdots&0\\
0&\alpha_2 & \cdots&0&\eta_2&0&0&\cdots&0\\
0&0&\ddots&0&\vdots&0&0&\cdots&0\\
0&0&\cdots&\alpha_{k-1}&\eta_{k-1}&0&0&\cdots&0\\
0&0&\cdots&0&\eta_{k}&0&0&\cdots&0\\
0&0&\cdots&0&\eta_{k+1}&\alpha_{k+1}&0&\cdots&0\\
0&0&\cdots&0&\eta_{k+2}&0&\alpha_{k+2}&\cdots&0\\
0&0&\cdots&0&\vdots&0&0&\ddots&0\\
0&0&\cdots&0&\eta_{m}&0&\cdots&0&\alpha_m\\
\end{pmatrix},
\end{equation}
where $\eta_1, \ldots, \eta_m\in\CZ^\circ$ and $\alpha_\ell\in\CZ^\times$ for $\ell\ne k$.  In particular, the determinant of $QA$ is \begin{equation*}
|QA|=\eta_k\prod_{\ell\ne k}\alpha_\ell.
\end{equation*}
\end{lemma}

\begin{proof}
Assuming $A$ has $m-1$ linearly independent rows, $m-1$ invertible elements can be placed along the main diagonal by row swapping.  Zeros can be obtained above and below each invertible element by adding invertible zeon multiples of rows together.   Assuming the $k$th row consists of nilpotent zeon elements, the $k$th column of the row reduced matrix consists of nilpotent elements as seen in \eqref{QA matrix}.  The determinant is computed by cofactor expansion, giving the result.  Using only the operations of swaps and addition of invertible zeon multiples of rows together, the determinant of $QA$ differs from that of $A$ by (at most) a sign; i.e., $|QA|=\pm |A|$.  Since all of the involved elementary row operations involve row swaps and addition of invertible zeon multiples, the resulting matrix $Q$, which is the product of the corresponding elementary matrices, is invertible.    
\end{proof}

\begin{definition}
Given a matrix $A\in \Mat(m;\CZ)$, a non-null $\xi\in{\CZ}^m$ is said to be a {\em zeon eigenvector} of $A$ if there exists a zeon eigenvalue $\lambda\in{\CZ}$ such that $A\xi=\lambda \xi$.
\end{definition}

Now suppose $\xi$ is a zeon eigenvector associated with zeon eigenvalue $\lambda$ and let $f(t)=a_kt^k+a_{k-1}t^{k-1}+\cdots+a_1 t+a_0$ be an arbitrary zeon polynomial.  The zeon matrix evaluation of $f$ then satisfies 
\begin{eqnarray*}
f(A)\xi&=&(a_kA^k+a_{k-1}A^{k-1}+\cdots+a_1 A+a_0\mI)\xi\\
&=&(a_k\lambda^k+a_{k-1}\lambda^{k-1}+\cdots+a_1 \lambda+a_0)\xi\\
&=&f(\lambda)\xi.
\end{eqnarray*}

\begin{theorem}\label{row reduction II}
Let $A\in\Mat(m;\CZ)$.  If $\lambda\in\CZ$ is a spectrally simple zero of $\chi_A(u)$, then there exists $\xi\in{\CZ}^m\setminus ({\CZ}^\circ)^m$ such that  $A\xi=\lambda\xi$.
\end{theorem}

\begin{proof}
Since $\lambda$ is a spectrally simple zero of the characteristic polynomial $\chi_A(u)$, the matrix $\lambda\mI-A$ is of rank one and thus, by applying Theorem \ref{row reduction I} and multiplying rows by element inverses, can be placed into the form 
\begin{equation*}
\rho(\lambda\mI-A)=\begin{pmatrix}
1&0 &\cdots&0&{\eta_1}&0&0&\cdots&0\\
0&1& \cdots&0&{\eta_2}&0&0&\cdots&0\\
0&0&\ddots&0&\vdots&0&0&\cdots&0\\
0&0&\cdots&1&{\eta_{k-1}}&0&0&\cdots&0\\
0&0&\cdots&0&\eta_k&0&0&\cdots&0\\
0&0&\cdots&0&{\eta_{k+1}}&1&0&\cdots&0\\
0&0&\cdots&0&{\eta_{k+2}}&0&1&\cdots&0\\
0&0&\cdots&0&\vdots&0&0&\ddots&0\\
0&0&\cdots&0&{\eta_{m}}&0&\cdots&0&1
\end{pmatrix},
\end{equation*}
where $\eta_1, \ldots, \eta_m\in\CZ^\circ$.  In particular, $\eta_k=\vert\rho(\lambda\mI-A)\vert=0$.  Solving the equation $(\lambda\mI-A)v=0$ now gives \begin{equation*}
v=\alpha\left(-{\eta_1},\cdots,-{\eta_{k-1}},1,-{\eta_{k+1}},\cdots,{\eta_m}\right)^\intercal,
\end{equation*}
satisfying $Av=\lambda v$ for any $\alpha\in\CZ$.

\end{proof}

\begin{corollary}
Let $A\in\Mat(m;\CZ)$.  If $\lambda$ is a zeon eigenvalue of $A$ associated with zeon eigenvector $\xi$,  then $\zco \lambda$ is an eigenvalue of $\zco A$ associated with eigenvector $\zco \xi$.
\end{corollary}

\begin{proof}
Expanding $A\xi=\lambda\xi$ as 
\begin{equation*}
(\zco A+\zdu A)(\zco \xi+\zdu \xi)=(\zco \lambda +\zdu \lambda)(\zco \xi +\zdu \xi),
\end{equation*}
we see that \begin{equation*}
\zco A\zco \xi+\zdu A\zco \xi+\zco A\zdu \xi+\zdu A\zdu \xi=\zco \lambda \zco \xi+\zdu \lambda \zco \xi+\zco \lambda\zdu \xi+\zdu \lambda\zdu \xi.
\end{equation*}
Since ${\CZ}^\circ$ is an ideal, it is easy to see that $\zco A\zco \xi$ and $\zco \lambda \zco \xi$ are the only terms of the equation that exist in $\mathbb{C}^m$.  All other terms are elements of $({\CZ}^\circ)^m$.  Hence, the result: $\zco A\zco \xi=\zco \lambda \zco \xi$.
\end{proof}

\begin{example}
Consider the matrix $A$ given by \begin{equation*}
A=\left(
\begin{array}{ccc}
 2+\zeta _{\{2\}} & \zeta _{\{3\}} & -\zeta _{\{1\}} \\
 \zeta _{\{3\}} & 3+\zeta _{\{1,2\}} & 4 \\
 -\zeta _{\{1\}} & 2 & 1 \\
\end{array}
\right).
\end{equation*}
The eigenvalues of $A$ are \begin{align*}
\lambda_1&=5+\frac{2}{3} \zeta _{\{1,2\}}-\frac{1}{3} \zeta _{\{1,3\}}-\frac{1}{9} \zeta _{\{1,2,3\}},\\
\lambda_2&=2+\frac{2}{3} \zeta _{\{1,3\}}+\zeta _{\{2\}}, { \rm and}\\
\lambda_3&=-1+\frac{1}{3} \zeta _{\{1,2\}}-\frac{1}{3} \zeta _{\{1,3\}}+\frac{1}{9} \zeta _{\{1,2,3\}}.
\end{align*}
Respective eigenvectors $v_1$, $v_2$, and $v_3$ associated $\lambda_1$, $\lambda_2$, and $\lambda_3$ are 
 \begin{align*}
 v_1&=\left(
\begin{array}{c}
 -\frac{1}{9} \zeta _{\{1,2\}}+\frac{2}{9} \zeta _{\{2,3\}}-\frac{1}{27} \zeta _{\{1,2,3\}}-\frac{\zeta _{\{1\}}}{3}+\frac{2 \zeta _{\{3\}}}{3} \vspace{3pt}\\
 2+\frac{1}{3} \zeta _{\{1,2\}}+\frac{1}{6} \zeta _{\{1,3\}}+\frac{1}{18} \zeta _{\{1,2,3\}}\vspace{3pt}\\
 1 \\
\end{array}
\right),\\
v_2&=\left(
\begin{array}{c}
 1 \\
 -\frac{1}{9} \zeta _{\{2,3\}}+\frac{1}{81} \zeta _{\{1,2,3\}}+\frac{4 \zeta _{\{1\}}}{9}-\frac{\zeta _{\{3\}}}{9}\vspace{3pt} \\
 \frac{1}{9} \zeta _{\{1,2\}}+\frac{2}{81} \zeta _{\{1,2,3\}}-\frac{\zeta _{\{1\}}}{9}-\frac{2 \zeta _{\{3\}}}{9}\vspace{3pt} \\
\end{array}
\right),\text{ \rm and}\\
v_3&=\left(
\begin{array}{c}
 -\frac{1}{9} \zeta _{\{1,2\}}-\frac{1}{9} \zeta _{\{2,3\}}-\frac{1}{54} \zeta _{\{1,2,3\}}+\frac{\zeta _{\{1\}}}{3}+\frac{\zeta _{\{3\}}}{3} \vspace{3pt}\\
 \frac{1}{6} \zeta _{\{1,2\}}-1 \vspace{3pt}\\
 1 \\
\end{array}
\right). 
\end{align*}
\end{example}

\section{The Zeon Spectral Theorem}\label{zeon spectral theorem}

In this section, linear independence of eigenvectors associated with distinct zeon eigenvalues is established and utilized to construct rank-one zeon projection operators.  These operators are used to construct a resolution of the identity and, ultimately, to establish a spectral theorem for zeon matrices. 

\begin{proposition}
Let $A\in \Mat(m;\CZ)$.  Let $v_1,v_2\in{\CZ}^m$ be zeon eigenvectors of $A$ associated with zeon eigenvalues $\lambda_1, \lambda_2$, respectively, where $\zco \lambda_1\ne \zco \lambda_2$.  Then, $v_1$ and $v_2$ are linearly independent.
\end{proposition}

\begin{proof}
Suppose, to the contrary, that for some nonzero constants $\alpha_1, \alpha_2\in\CZ$, the following holds:\begin{align}
\zm&=\alpha_1v_1+\alpha_2v_2\nonumber\\
&=(\alpha_1v_1+\alpha_2v_2)\lambda_1.\label{orth1}
\end{align}
Further, linearity of $A$ implies
\begin{align}
\zm&=A(\alpha_1v_1+\alpha_2v_2)\nonumber\\
&=\alpha_1\lambda_1v_2+\alpha_2\lambda_2v_2.\label{orth2}
\end{align}
Subtracting \eqref{orth2} from \eqref{orth1} gives
\begin{align*}
\zm&=\alpha_2(\lambda_2-\lambda_1)v_2.
\end{align*}
Since, $\zco\lambda_1\ne\zco\lambda_2$, we conclude that $\lambda_2-\lambda_1$ is invertible.  Hence, $\alpha_2 v_2=\zm$, which further implies $\alpha_1v_1=\zm$ by \eqref{orth1} so that $v_1,v_2\in({\CZ}^\circ)^m$, a contradiction.   
\end{proof}

\begin{remark}
Note that the only zeon eigenvalue that can be associated with the nullspace of matrix $A\in\Mat(m;\CZ)$ is zero.  If $\lambda v=0$ for nonzero $\lambda$, then $\lambda$ is nilpotent and $v\in({\CZ}^\circ)^m$, violating the requirement that an individual eigenvector be linearly independent.  
\end{remark}

The zeon eigenvalues of a matrix are restricted to spectrally simple zeros of the matrix's characteristic polynomial.  Suppose $v$ is an eigenvector of $X$ associated with distinct eigenvalues $\lambda_1, \lambda_2$ satisfying $\zco\lambda_1=\zco\lambda_2$.  Then, 
\begin{align*}
Xv=\lambda_1v&=\lambda_2 v\\
&\Rightarrow (\lambda_2-\lambda_1)v=\zm,
\end{align*}

\subsection{The Inner Product on $\CZ^m$}\label{zeon inner product}

Let $x,y\in{\CZ}^m$.  Writing $x$ and $y$ as column matrices $y=(y_1, \ldots, y_m)^\intercal$ and \hfill\break $x=(x_1, \ldots, x_m)^\intercal$, the {\em zeon module inner product} of $x$ and $y$ is defined by
\begin{equation}\label{inner prod def}
\langle x,y\rangle = y^\dag x,
\end{equation}
where $y^\dag=\overline{y}^\intercal$ denotes the complex conjugate transpose of $y$.

Some basic properties are established in the next lemma, the proof of which is straightforward.

\begin{lemma}
Let  $x,y,z\in{\CZ}^m$ and let $\alpha\in\CZ$.  Then,
\begin{align}
\langle \alpha x+y, z\rangle &= \alpha \langle x, z\rangle+ \langle y,z\rangle\label{firstlinearity}\\  
\langle x,y\rangle &= \overline{\langle y, x\rangle}\\  
\zco \langle x,y\rangle &=\langle \zco x, \zco y\rangle\\
\zco\langle x,x\rangle &\ge 0\\
\zco\langle x,x\rangle &= 0 \text{ \rm iff }x\in({\CZ}^\circ)^m.
\end{align} 
\end{lemma}

The quantity $\langle x,x\rangle$ is generally not scalar-valued, so it does not define a norm.  However, the scalar part of $\langle x,x\rangle$ defines a seminorm.   Since $\zco \langle x,x\rangle\ge 0$ for all $x\in{\CZ}^m$, we define the {\em spectral seminorm} of $x$ by \begin{equation*}
|x|_\star=\zco\langle x, x\rangle^{1/2}.
\end{equation*}

\begin{definition}
A zeon vector $v\in{\CZ}^m$ is said to be {\em null} if its spectral seminorm is zero, i.e., $\zco\langle v, v\rangle = 0$, or equivalently, $\langle v,v\rangle\in{\CZ}^\circ$.  A non-null zeon vector $v$ is said to be {\em normalized} if and only if $\langle v,v\rangle =1$.  
\end{definition}

Any non-null $x\in{\CZ}^m$ can be {\em normalized} via the mapping $x\mapsto \hat{x}$ where \begin{equation*}
\hat{x}=(\langle x, x\rangle^{-1})^{1/2} x.
\end{equation*}
It is not difficult to verify that $\langle \hat{x}, \hat{x}\rangle=1$ whenever $\zco \langle x,x\rangle\ne 0$.

\begin{example}
Consider the following elements of ${\CZ}^3$:
\begin{align*}
v_1=\begin{pmatrix} i+\zeta _{\{1\}}\\
 \zeta _{\{2\}}-\zeta _{\{2,3\}}\\
 2-\zeta _{\{1,2,3\}}
 \end{pmatrix}; &&
v_2=\begin{pmatrix}
 \zeta _{\{1,3\}}\\
 \zeta _{\{2\}}\\
 \frac{i}{2}  \zeta _{\{1,3\}}\end{pmatrix}.
\end{align*}

Direct calculation shows that \begin{align*}
\langle v_1,v_1\rangle&=5-4 \zeta _{\{1,2,3\}}\\
\langle v_1,v_2\rangle&=0\\
\langle v_2,v_2\rangle&=0.
\end{align*}

While $v_2$ is null (and can't be normalized), $v_1$ can be normalized by mapping \begin{equation*}
v_1\mapsto w_1=\langle v_1, v_1\rangle^{-1/2} v_1,
\end{equation*}
where \begin{equation*}
\langle v_1,v_1\rangle^{-1/2}=\frac{1}{\sqrt{5}}+\frac{2}{5 \sqrt{5}}\zeta _{\{1,2,3\}}.
\end{equation*}
In this case, the normalized vector $w_1$ is \begin{equation*}
w_1=\left(
\begin{array}{c}
\frac{i}{\sqrt{5}}+\frac{1}{\sqrt{5}}\zeta _{\{1\}}+ \frac{2 i}{5 \sqrt{5}}\zeta _{\{1,2,3\}}\\
 \frac{1}{\sqrt{5}}\zeta _{\{2\}}-\frac{1}{\sqrt{5}}\zeta _{\{2,3\}} \\
 \frac{2}{\sqrt{5}}-\frac{1}{5 \sqrt{5}}\zeta _{\{1,2,3\}} \\
\end{array}
\right).
\end{equation*}
Finally, direct computation verifies that $\langle w_1, w_1\rangle  =1$.
\end{example}

\subsection*{Resolution of the Identity}

Two vectors $v_1,v_2\in{\CZ}^m$ are said to be {\em orthogonal} if and only if $\langle v_1, v_2\rangle=0$.   Note that any collection of zeon vectors $\{v_1, \ldots, v_m\}$ spanning ${\CZ}^m$ can be orthogonalized via Gaussian elimination.  Normalizing then yields an orthonormal basis for ${\CZ}^m$.

\begin{lemma}
Let $v\in{\CZ}^m$ be a normalized zeon vector.  The matrix $vv^\dag$ represents orthogonal projection onto $\spn(\{v\})$.
\end{lemma}

\begin{proof}
Given normalized $v$, let $\{u_1, \ldots u_{m-1}\}$ be an orthonormalized collection of zeon vectors orthogonal to $v$.  Letting $x\in{\CZ}^m$ be arbitrary, there exist zeon coefficients $\alpha_0, \ldots, \alpha_{m-1}$ such that \begin{equation*}
x=\alpha_0 v+\alpha_1 u_1+\cdots+\alpha_{m-1}u_{m-1}.
\end{equation*}
It follows that \begin{align*}
(vv^\dag)x&=vv^\dag(\alpha_0 v+\alpha_1 u_1+\cdots+\alpha_{m-1}u_{m-1})\\
&=v \alpha_0 v^\dag v+\alpha_1 v^\dag u_1+\cdots+\alpha_{m-1}v^\dag u_{m-1}\\
&=\alpha_0 \langle v,v\rangle v+\alpha_1\langle u_1,v \rangle v+\cdots+\alpha_{m-1}\langle u_{m-1},v\rangle v\\
&=\alpha_0 v.
\end{align*} 
\end{proof}

It follows that when $\{u_1, \ldots, u_m\}$ is an orthonormalized collection of zeon vectors, a {\em resolution of the identity} is given by \begin{equation*}
\mI=\bigoplus_{j=1}^m u_j{u_j}^\dag.
\end{equation*}

A matrix $A\in\Mat(m;\CZ)$ satisfying $A^\dag=A$ is clearly self-adjoint w.r.t. the inner product \eqref{inner prod def} since \begin{align*}
\langle Ax, y\rangle=y^\dag (Ax)=y^\dag(x^\dag A^\dag)^\dag&=y^\dag(x^\dag A)^\dag\\
&=y^\dag A^\dag x=(Ay)^\dag x=\langle x, Ay\rangle
\end{align*}
for all $x,y\in{\CZ}^m$. 

Eigenvalues of self-adjoint zeon matrices are elements of the real zeon algebra $\RZ$; i.e., $\chi_A(\lambda)=0$ implies $\lambda=\overline{\lambda}$ when $A$ is self-adjoint.   Further, eigenvectors of self-adjoint zeon matrices associated with distinct eigenvalues are orthogonal.  Given distinct eigenvalues $\lambda_1$ and $\lambda_2$ associated with eigenvectors $v_1$ and $v_2$, respectively, 
\begin{align*}
\lambda_1\langle v_1, v_2\rangle&=\langle Av_1, v_2\rangle\\
&=\langle v_1, A v_2\rangle\\
&=\lambda_2 \langle v_1, v_2\rangle
\end{align*}
implies $\langle v_1, v_2\rangle=0$.

\begin{theorem}[Zeon Spectral Theorem]\label{strong spectral theorem}
Let $A\in\Mat(m;\CZ)$ be a self-adjoint zeon matrix with $m$ spectrally simple eigenvalues.  Let $v_1, \ldots, v_m$ denote normalized zeon eigenvectors associated with these eigenvalues and set $\pi_j=v_j{v_j}^\dag$ for $j=1, \ldots, m$.  Then, \begin{equation*}
A=\bigoplus_{j=1}^m \lambda_j\pi_j.
\end{equation*}
\end{theorem}

\begin{proof}
Assuming that $A$ is self-adjoint with spectrally simple eigenvalues, it follows that the corresponding normalized eigenvectors $\{v_1, \ldots, v_m\}$ are orthogonal, resulting in an orthonormal zeon basis of ${\CZ}^m$.  For arbitrary $z\in {\CZ}^m$, it follows that \begin{align*}
Az=A\sum_{j=1}^m \langle z,v_j\rangle v_j=\sum_{j=1}^m \lambda_j \langle z,v_j\rangle v_j&=\sum_{j=1}^m \lambda_j \pi_j z\\
&=\left(\bigoplus_{j=1}^m \lambda_j \pi_j\right)z.
\end{align*}
\end{proof}

\begin{example}\label{spectral decomp example}
Consider the self-adjoint zeon matrix
\begin{equation*}
A=\left(
\begin{array}{ccc}
 5+\zeta _{\{2\}} & \zeta _{\{3\}} & -\zeta _{\{1\}} \\
 \zeta _{\{3\}} & 6+\zeta _{\{1,2\}} & 4 \\
 -\zeta _{\{1\}} & 4 & 6 \\
\end{array}
\right).
\end{equation*}
The zeon eigenvalues of $A$ are \begin{align*}
\lambda_1&=10+\frac{1}{2} \zeta _{\{1,2\}}-\frac{1}{5} \zeta _{\{1,3\}}-\frac{1}{25} \zeta _{\{1,2,3\}},\\
\lambda_2&=5+\zeta _{\{2\}}+\frac{8}{15} \zeta _{\{1,3\}}-\frac{16}{225} \zeta _{\{1,2,3\}},\\
\lambda_3&=2+\frac{1}{2} \zeta _{\{1,2\}}-\frac{1}{3} \zeta _{\{1,3\}}+\frac{1}{9} \zeta _{\{1,2,3\}}.
\end{align*}

The matrices $\{\lambda_j\mI-A: j=1,2,3\}$  reduce via elementary row operations to the following: \begin{align*}
\lambda_1\mI-A&\mapsto\left(
\begin{array}{ccc}
 1 & 0 & \frac{1}{5}\zeta _{\{1\}}-\frac{1}{5}\zeta _{\{3\}} +\frac{1}{25} \zeta _{\{1,2\}}-\frac{1}{25} \zeta _{\{2,3\}}-\frac{1}{200} \zeta _{\{1,2,3\}}\vspace{3pt}\\
 0 & 1 & -1-\frac{1}{8} \zeta _{\{1,2\}}\vspace{3pt} \\
 0 & 0 & 0 \\
\end{array}
\right),\\
\lambda_2\mI-A&\mapsto\left(
\begin{array}{ccc}
 -\frac{4}{15} \zeta _{\{1\}}-\frac{1}{15}\zeta _{\{3\}}+\frac{8}{225} \zeta _{\{1,2\}}+\frac{17}{225} \zeta _{\{2,3\}}-\frac{1}{225} \zeta _{\{1,2,3\}} & 1 & 0 \vspace{3pt}\\
\frac{1}{15}\zeta _{\{1\}}+\frac{4}{15} \zeta _{\{3\}}  -\frac{17}{225} \zeta _{\{1,2\}}-\frac{8}{225} \zeta _{\{2,3\}}+\frac{4}{225} \zeta _{\{1,2,3\}}& 0 & 1 \vspace{3pt}\\
 0 & 0 & 0 
\end{array}
\right),\\
\lambda_3\mI-A&\mapsto\left(
\begin{array}{ccc}
 1 & \frac{1}{3}\zeta _{\{1\}}+\frac{1}{3}\zeta _{\{3\}}-\frac{1}{9} \zeta _{\{1,2\}}-\frac{1}{9} \zeta _{\{2,3\}}+\frac{1}{18} \zeta _{\{1,2,3\}} & 0 \vspace{3pt}\\
 0 & 1+\frac{1}{8} \zeta _{\{1,2\}}& 1 \vspace{3pt}\\
 0 & 0 & 0 \\
\end{array}
\right),
\end{align*}
giving us the corresponding (non-normalized) zeon eigenvectors:  
\begin{align*}
v_1&=\left(
\begin{array}{c}
 -\frac{1}{5}\zeta _{\{1\}}+\frac{1}{5}\zeta _{\{3\}}-\frac{1}{25} \zeta _{\{1,2\}}+\frac{1}{25} \zeta _{\{2,3\}}+\frac{1}{200} \zeta _{\{1,2,3\}} \vspace{3pt}\\
1+ \frac{1}{8} \zeta _{\{1,2\}}\vspace{3pt}\\
 1 \\
\end{array}
\right),\\
v_2&=\left(
\begin{array}{c}
 1 \\
 \vspace{3pt}
 \frac{4}{15} \zeta _{\{1\}}+\frac{1}{15}\zeta _{\{3\}} -\frac{8}{225} \zeta _{\{1,2\}}-\frac{17}{225} \zeta _{\{2,3\}}+\frac{1}{225} \zeta _{\{1,2,3\}}\vspace{3pt}\\ 
 -\frac{1}{15}\zeta _{\{1\}}-\frac{4}{15} \zeta _{\{3\}}+\frac{17}{225} \zeta _{\{1,2\}}+\frac{8}{225} \zeta _{\{2,3\}}-\frac{4}{225} \zeta _{\{1,2,3\}} \vspace{3pt}\\
\end{array}
\right),\text{ \rm and}\\
v_3&=\left(
\begin{array}{c}
 -\frac{1}{3}\zeta _{\{1\}}-\frac{1}{3}\zeta _{\{3\}}+\frac{1}{9} \zeta _{\{1,2\}}+\frac{1}{9} \zeta _{\{2,3\}}-\frac{1}{18} \zeta _{\{1,2,3\}} \vspace{3pt}\\
 1 \vspace{3pt}\\
-1 -\frac{1}{8} \zeta _{\{1,2\}} \\
\end{array}
\right).
\end{align*}

The normalized zeon eigenvectors $\hat{v_j}=\langle v_j, v_j\rangle^{-1/2}v_j$ are \begin{align*}
\hat{v_1}&=\left(
\begin{array}{c}
-\frac{1}{5 \sqrt{2}}\zeta _{\{1\}}+\frac{1}{5 \sqrt{2}}\zeta _{\{3\}} + \frac{1}{25 \sqrt{2}}\zeta _{\{2,3\}}-\frac{1}{25 \sqrt{2}}\zeta _{\{1,2\}}-\frac{3}{400 \sqrt{2}}\zeta _{\{1,2,3\}}\vspace{3pt}\\
 \frac{1}{\sqrt{2}}+\frac{1}{16 \sqrt{2}}\zeta _{\{1,2\}}+\frac{1}{50 \sqrt{2}}\zeta _{\{1,3\}}+\frac{1}{125 \sqrt{2}}\zeta _{\{1,2,3\}} \vspace{3pt}\\
\frac{1}{\sqrt{2}} -\frac{1}{16 \sqrt{2}}\zeta _{\{1,2\}}+\frac{1}{50 \sqrt{2}}\zeta _{\{1,3\}}+\frac{1}{125 \sqrt{2}}\zeta _{\{1,2,3\}} 
\end{array}
\right),\\
\hat{v_2}&=\left(
\begin{array}{c}
1 -\frac{8}{225}\zeta _{\{1,3\}}+\frac{152}{3375}\zeta _{\{1,2,3\}} \vspace{3pt}\\
 \frac{4}{15} \zeta _{\{1\}}+\frac{1}{15}\zeta _{\{3\}}-\frac{8}{225} \zeta _{\{1,2\}}-\frac{17}{225} \zeta _{\{2,3\}}+\frac{1}{225} \zeta _{\{1,2,3\}} \vspace{3pt}\\
 -\frac{1}{15}\zeta _{\{1\}}-\frac{4}{15} \zeta _{\{3\}}+\frac{17}{225} \zeta _{\{1,2\}}+\frac{8}{225} \zeta _{\{2,3\}}-\frac{4}{225} \zeta _{\{1,2,3\}} 
\end{array}
\right),\text{ \rm and}\\
\hat{v_3}&=\left(
\begin{array}{c}
 -\frac{1}{3 \sqrt{2}}\zeta _{\{1\}}-\frac{1}{3 \sqrt{2}}\zeta _{\{3\}}+\frac{1}{9 \sqrt{2}}\zeta _{\{1,2\}}+\frac{1}{9 \sqrt{2}}\zeta _{\{2,3\}}-\frac{5}{144 \sqrt{2}}\zeta _{\{1,2,3\}}\vspace{3pt} \\
 \frac{1}{\sqrt{2}}-\frac{1}{16 \sqrt{2}}\zeta _{\{1,2\}}-\frac{1}{18 \sqrt{2}}\zeta _{\{1,3\}} +\frac{1}{27 \sqrt{2}}\zeta _{\{1,2,3\}}\vspace{3pt}\\
 -\frac{1}{\sqrt{2}} -\frac{1}{16 \sqrt{2}}\zeta _{\{1,2\}}+\frac{1}{18 \sqrt{2}}\zeta _{\{1,3\}}-\frac{1}{27 \sqrt{2}}\zeta _{\{1,2,3\}}
\end{array}
\right).
\end{align*}

For $j=1,2,3$, the orthogonal zeon projections are determined by setting \begin{equation*}
\pi_j=\langle v_j,v_j\rangle^{-1} v_j{v_j}^\dag=\hat{v_j}\hat{v_j}^\dag.
\end{equation*}  
The resulting matrices are notationally cumbersome. The first one is presented here as  $\pi_1=(p_1\mid p_2\mid p_3)$, where 

\begin{align*}
p_1&=\left(
\begin{array}{c}
 -\frac{1}{25} \zeta _{\{1,3\}}-\frac{2}{125} \zeta _{\{1,2,3\}} \vspace{3pt}\\
 -\frac{1}{10}\zeta _{\{1\}}+\frac{1}{10}\zeta _{\{3\}}-\frac{1}{50} \zeta _{\{1,2\}}+\frac{1}{50} \zeta _{\{2,3\}}+\frac{1}{400} \zeta _{\{1,2,3\}} \vspace{3pt}\\
 -\frac{1}{10}\zeta _{\{1\}}+\frac{1}{10}\zeta _{\{3\}}-\frac{1}{50} \zeta _{\{1,2\}}+\frac{1}{50} \zeta _{\{2,3\}}-\frac{1}{100} \zeta _{\{1,2,3\}} 
\end{array}
\right),\\
p_2&=\left(
\begin{array}{c}
 -\frac{1}{10}\zeta _{\{1\}}+\frac{1}{10}\zeta _{\{3\}}-\frac{1}{50} \zeta _{\{1,2\}}+\frac{1}{50} \zeta _{\{2,3\}}+\frac{1}{400} \zeta _{\{1,2,3\}} \vspace{3pt}\\
 \frac{1}{2}+\frac{1}{16} \zeta _{\{1,2\}}+\frac{1}{50} \zeta _{\{1,3\}}+\frac{1}{125} \zeta _{\{1,2,3\}}\vspace{3pt} \\
 \frac{1}{2}+\frac{1}{50} \zeta _{\{1,3\}}+\frac{1}{125} \zeta _{\{1,2,3\}} \vspace{3pt}
\end{array}
\right),\text{ \rm and}\\
p_3&=\left(
\begin{array}{c}
-\frac{1}{10}\zeta _{\{1\}}+\frac{1}{10}\zeta _{\{3\}} -\frac{1}{50} \zeta _{\{1,2\}}+\frac{1}{50} \zeta _{\{2,3\}}-\frac{1}{100} \zeta _{\{1,2,3\}} \vspace{3pt}\\
 \frac{1}{2}+\frac{1}{50} \zeta _{\{1,3\}}+\frac{1}{125} \zeta _{\{1,2,3\}} \vspace{3pt}\\
\frac{1}{2}  -\frac{1}{16} \zeta _{\{1,2\}}+\frac{1}{50} \zeta _{\{1,3\}}+\frac{1}{125} \zeta _{\{1,2,3\}}
\end{array}
\right).
\end{align*}

The projection corresponding to eigenvector $v_2$ is $\pi_2=(q_1\vert q_2\vert q_3)$, where 
\begin{align*}
q_1&=\left(
\begin{array}{c}
 1-\frac{16}{225} \zeta _{\{1,3\}}+\frac{304}{3375} \zeta _{\{1,2,3\}}\vspace{3pt} \\
 \frac{4}{15} \zeta _{\{1\}}+\frac{1}{15}\zeta _{\{3\}}-\frac{8}{225} \zeta _{\{1,2\}}-\frac{17}{225} \zeta _{\{2,3\}}+\frac{1}{225} \zeta _{\{1,2,3\}} \vspace{3pt}\\
 -\frac{1}{15}\zeta _{\{1\}}-\frac{4}{15} \zeta _{\{3\}}+\frac{17}{225} \zeta _{\{1,2\}}+\frac{8}{225} \zeta _{\{2,3\}}-\frac{4}{225} \zeta _{\{1,2,3\}}\vspace{3pt} 
\end{array}
\right),\\
q_2&=\left(
\begin{array}{c}
\frac{4}{15} \zeta _{\{1\}}+\frac{1}{15}\zeta _{\{3\}} -\frac{8}{225} \zeta _{\{1,2\}}-\frac{17}{225} \zeta _{\{2,3\}}+\frac{1}{225} \zeta _{\{1,2,3\}} \vspace{3pt}\\
 \frac{8}{225} \zeta _{\{1,3\}}-\frac{152}{3375} \zeta _{\{1,2,3\}} \vspace{3pt}\\
 -\frac{17}{225} \zeta _{\{1,3\}} +\frac{98}{3375} \zeta _{\{1,2,3\}}
\end{array}
\right),\text{ \rm and}\\
q_3&=\left(
\begin{array}{c}
 -\frac{1}{15}\zeta _{\{1\}}-\frac{4}{15} \zeta _{\{3\}}+\frac{17}{225} \zeta _{\{1,2\}}+\frac{8}{225} \zeta _{\{2,3\}}-\frac{4}{225} \zeta _{\{1,2,3\}} \vspace{3pt}\\
 -\frac{17}{225} \zeta _{\{1,3\}}+\frac{98}{3375} \zeta _{\{1,2,3\}} \vspace{3pt}\\
 \frac{8}{225} \zeta _{\{1,3\}}-\frac{152}{3375} \zeta _{\{1,2,3\}} 
\end{array}
\right).
\end{align*}

Finally, the projection corresponding to eigenvector $v_3$ is $\pi_3=(s_1\vert s_2\vert s_3)$, where 
\begin{align*}
s_1&=\left(
\begin{array}{c}
 \frac{1}{9} \zeta _{\{1,3\}}-\frac{2}{27} \zeta _{\{1,2,3\}}\vspace{3pt} \\
 -\frac{1}{6}\zeta _{\{1\}}-\frac{1}{6}\zeta _{\{3\}}+\frac{1}{18} \zeta _{\{1,2\}}+\frac{1}{18} \zeta _{\{2,3\}}-\frac{1}{144} \zeta _{\{1,2,3\}} \vspace{3pt}\\
 \frac{1}{6}\zeta _{\{1\}}+\frac{1}{6}\zeta _{\{3\}}-\frac{1}{18} \zeta _{\{1,2\}}-\frac{1}{18} \zeta _{\{2,3\}}+\frac{1}{36} \zeta _{\{1,2,3\}} 
\end{array}
\right),\\
s_2&=\left(
\begin{array}{c}
 -\frac{1}{6}\zeta _{\{1\}}-\frac{1}{6}\zeta _{\{3\}}+\frac{1}{18} \zeta _{\{1,2\}}+\frac{1}{18} \zeta _{\{2,3\}}-\frac{1}{144} \zeta _{\{1,2,3\}} \vspace{3pt}\\
 \frac{1}{2}-\frac{1}{16} \zeta _{\{1,2\}}-\frac{1}{18} \zeta _{\{1,3\}}+\frac{1}{27} \zeta _{\{1,2,3\}} \vspace{3pt}\\
 -\frac{1}{2}+\frac{1}{18} \zeta _{\{1,3\}}-\frac{1}{27} \zeta _{\{1,2,3\}} 
\end{array}
\right),\text{ \rm and}\\
s_3&=\left(
\begin{array}{c}
\frac{1}{6}\zeta _{\{1\}}+\frac{1}{6}\zeta _{\{3\}} -\frac{1}{18} \zeta _{\{1,2\}}-\frac{1}{18} \zeta _{\{2,3\}}+\frac{1}{36} \zeta _{\{1,2,3\}} \vspace{3pt}\\
 -\frac{1}{2}+\frac{1}{18} \zeta _{\{1,3\}}-\frac{1}{27} \zeta _{\{1,2,3\}} \vspace{3pt}\\
\frac{1}{2}+ \frac{1}{16} \zeta _{\{1,2\}}-\frac{1}{18} \zeta _{\{1,3\}}+\frac{1}{27} \zeta _{\{1,2,3\}} 
\end{array}
\right).
\end{align*}

The following properties can  be verified by direct computation:
\begin{enumerate}[i.]
\item ${\pi_j}^2=\pi_j$ for $j=1, 2, 3$ (idempotent);
\item $\pi_j\pi_k=\pi_k\pi_j=\zm$ for $j\ne k$ (orthogonality);
\item $\pi_1+\pi_2+\pi_3=\mI_3$ (resolution of the identity); and 
\item $A=\lambda_1\pi_1+\lambda_2\pi_2+\lambda_3\pi_3$ (spectral decomposition).
\end{enumerate}
\end{example}

\section{Zeon Matrix Exponentials as Finite Sums}\label{zeon matrix exponential}

Spectrally simple zeon matrices lend themselves nicely to evaluation of matrix exponentials.  In this section, we consider a number of cases in which exponentials of zeon matrices can be computed as finite sums.  For clarity the extension of the exponential function to $\CZ$ is denoted by $\zexp:\CZ\to \CZ$, which extends in the expected way to matrices in $\Mat(m;\CZ)$.

As first seen in \cite{staplesweygandt}, the exponential of $\lambda=\zco\lambda+\zdu \lambda\in\CZ$ is computed by \begin{equation*}
\zexp(\lambda)=\sum_{\ell=0}^{\kappa(\zdu \lambda)-1}\frac{\exp(\zco \lambda)}{\ell!}(\zdu \lambda)^\ell
=\exp(\zco \lambda)\sum_{\ell=0}^{\kappa(\zdu\lambda)-1}\frac{(\zdu \lambda)^\ell}{\ell!}.
\end{equation*}

Given $A\in\Mat(m,\CZ)$, the exponential of $A$ is defined by   
\begin{equation}\label{matrix exp def}
\zexp(A)=\sum_{k=0}^\infty A^k/k!=\lim_{n\to\infty}\sum_{k=0}^n (\zco A+\zdu A)^k/k!,
\end{equation}
provided that the limit exists (relative to some norm\footnote{A detailed study of zeon matrix norms and norm inequalities is anticipated as a sequel to the current paper.} on $\Mat(m;\CZ)$).    

\subsection{Self-Adjoint Matrices}

If $A\in\Mat(m;\CZ)$ is spectrally simple and self-adjoint, the zeon spectral theorem implies \begin{align*}
\zexp(A)=\sum_{\ell=0}^{\infty}\frac{A^\ell}{\ell!}=\sum_{\ell=0}^{\infty}\frac{1}{\ell!}\left(\bigoplus_{j=1}^{m}\lambda_j\pi_j\right)^\ell&=\sum_{\ell=0}^{\infty}\frac{1}{\ell!}\bigoplus_{j=1}^{m}{\lambda_j}^\ell\pi_j\\
&=\bigoplus_{j=1}^m\pi_j\sum_{\ell=0}^\infty\frac{{\lambda_j}^\ell}{\ell!}=\bigoplus_{j=1}^m \zexp(\lambda_j)\pi_j,
\end{align*}
where $\pi_j$ is the rank-one projector onto the eigenspace of $\lambda_j$ for $j=1, \ldots, m$.

\begin{example}
Recalling the self-adjoint zeon matrix $A$ from Example \ref{spectral decomp example}, we see that \begin{align*}
\zexp(\lambda_1)&=e^{10}\left(1+\frac{1}{2} \zeta _{\{1,2\}}-\frac{1}{5} \zeta _{\{1,3\}}-\frac{1}{25} \zeta _{\{1,2,3\}}\right),\\
\zexp(\lambda_2)&=e^5\left(1+\zeta_{\{2\}}+\frac{8}{15}  \zeta _{\{1,3\}}+\frac{104}{225}  \zeta _{\{1,2,3\}}\right),\\
\zexp(\lambda_3)&=e^2\left(1+\frac{1}{2} \zeta _{\{1,2\}}-\frac{1}{3} \zeta _{\{1,3\}}+\frac{1}{9} \zeta _{\{1,2,3\}}\right).
\end{align*}
Now $\zexp(A)=\zexp(\lambda_1)\pi_1\oplus \zexp(\lambda_2)\pi_2\oplus \zexp(\lambda_3)\pi_3$, where $\pi_1$, $\pi_2$, and $\pi_3$ are the rank-one projections given in the example.

\end{example}
 
\subsection{Nilpotent Matrices}
When $A$ is nilpotent, the zeon matrix exponential reduces to a finite sum of the form
\begin{equation*}
\zexp(A)=\sum_{\ell=0}^{\kappa(A)}\frac{A^\ell}{\ell!}.
\end{equation*} 

Observing that $\zco (A^k)=(\zco A)^k$ for any $k\in \mathbb{N}$ and that $\zdu A$ is always nilpotent, it follows that $A$ is nilpotent if and only if $\zco A$ is nilpotent.

\subsubsection{Application: The nilpotent adjacency matrix of a graph}

Given a simple graph $G=(V,E)$ on vertices $V=\{1,\ldots, n\}$ and edges $E$ consisting of unordered pairs of vertices, the nilpotent adjacency matrix of $G$ is the $n\times n$ matrix $\Psi=(\psi_{ij})$ whose entries are defined by \begin{equation*}
\psi_{ij}=\begin{cases}
\zeta_{\{j\}}& \{i,j\}\in E,\\
0 & \{i,j\}\notin E.
\end{cases}
\end{equation*}
Note that because the maximum path length in a graph on $n$ vertices is $n$, the index of nilpotency of $\Psi$ is at most $n+1$; i.e., $\kappa(\Psi)\le n+1$.  As seen in \cite{ICCA7}, the $k$th power of $\Psi$ can be used to count paths and cycles of length $k$ in the graph.  The matrix exponential then counts paths and cycles of all lengths in the graph.  In particular, \begin{equation*}
\langle i\vert \zexp(\Psi)\vert i\rangle =\sum_{\{I\in 2^{[n]}\}}\frac{\alpha_I}{|I|!}\zeta_I,
\end{equation*}
where $\alpha_I$ is the number of cycles on vertices indexed by $I$ based at vertex $i$.  Further, for $i\ne j$, taking $\langle \zeta_{\{i\}}\vert$ to mean the row vector with $\zeta_{\{i\}}$ in the $i$th position and zeros elsewhere, 
\begin{equation*}
\langle \zeta_{\{i\}}\vert \zexp(\Psi)\vert j\rangle =\sum_{\{I\in 2^{[n]}\}}\frac{\alpha_I}{(|I|-1)!}\zeta_I,
\end{equation*}
where $\alpha_I$ denotes the number of paths from $i$ to $j$ on vertices indexed by $I$.

\begin{example}
\begin{figure}
$$\includegraphics[width=160pt]{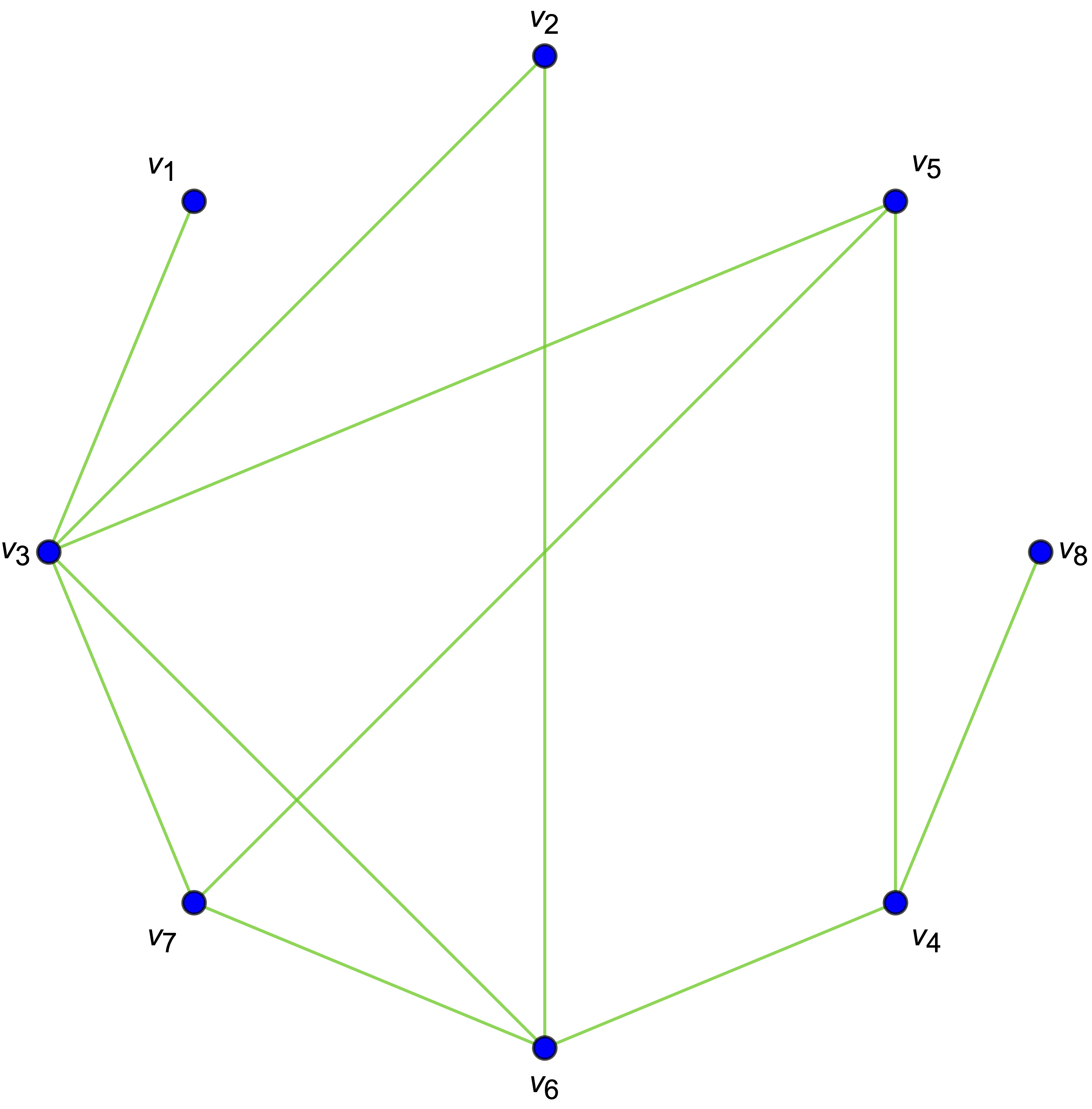}\hspace{10pt}
\includegraphics[width=160pt]{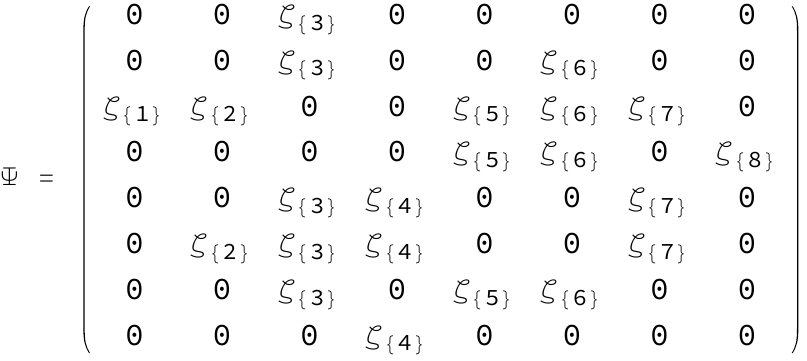}$$
\caption{An $8$-vertex graph and its nilpotent adjacency matrix.\label{8vgraph}}
\end{figure}
In the graph of Figure \ref{8vgraph}, cycles based at vertex $4$ are revealed by the corresponding diagonal entry of $\zexp(\Psi)$:
\begin{align*}
\left\langle 4\left|\zexp(\Psi)\right|4\right\rangle &=1+\frac{1}{2} \zeta _{\{4,5\}}+\frac{1}{2} \zeta _{\{4,6\}}+\frac{1}{2} \zeta _{\{4,8\}}+\frac{1}{12} \zeta _{\{3,4,5,6\}}+\frac{1}{12} \zeta _{\{4,5,6,7\}}\\
&+\frac{1}{60} \zeta _{\{2,3,4,5,6\}}+\frac{1}{30} \zeta _{\{3,4,5,6,7\}}+\frac{1}{360} \zeta _{\{2,3,4,5,6,7\}}.
\end{align*}
All paths from vertex $4$ to vertex $7$ are revealed by the entry in row-4, column-7:
\begin{align*}
\left\langle \zeta _{\{4\}}\left|\zexp(\Psi)\right|7\right\rangle&= \frac{1}{2} \zeta _{\{4,6,7\}}+\frac{1}{6} \zeta _{\{1,4,7,8\}}+\frac{1}{24} \zeta _{\{4,5,6,7,8\}}+\frac{1}{120} \zeta _{\{1,4,5,6,7,8\}}.
\end{align*}
\end{example}

\subsection{Matrices $A\in\Mat(m;\CZ)$ satisfying $\zco A\zdu A=\zdu A\zco A$}

When $\zco A$ and $\zdu A$ commute, the zeon matrix exponential is reduced to a finite sum involving the scalar matrix exponential $\exp(\zco A)$.  In particular, 
\begin{align}\label{finite sum comm case}
\zexp(A)=\zexp(\zco A+\zdu A)&=\exp(\zco A)\zexp(\zdu A)\nonumber\\
&=\exp(\zco A)\sum_{\ell=0}^{\kappa(\zdu A)}\frac{(\zdu A)^\ell}{\ell!},
\end{align}
where the scalar matrix exponential $\exp(\zco A)$ remains an infinite series, computed by  ordinary methods.   In some cases, the exponential can be reduced to a truly finite sum, as seen below.

\subsubsection{Application: The zeon Laplacian of a degree-regular graph}

The zeon combinatorial Laplacian of a graph $G$ on $n$ vertices is an $n\times n$ matrix $\Lambda=D-\Psi$, where $D$ is a diagonal matrix of vertex degrees and $\Psi$ is the nilpotent adjacency matrix of $G$.   Here, $\zco\Lambda=D$ and $\zdu\Lambda=-\Psi$.

As was shown in previous work, when $\Lambda$ is the zeon Laplacian of {\em any} finite graph with no isolated vertices, the inverse $\Lambda^{-1}$ enumerates all paths in the graph \cite{zeon laplacian}.   

When $\Lambda$ is the zeon combinatorial Laplacian of an $r$-regular graph $G$ on $n$ vertices,  $\Lambda=D-\Psi$, where $D=r\mI_n$ and $\Psi$ is the nilpotent adjacency matrix of $G$.   In this case, it is clear that $D\Psi=\Psi D$ so that the matrix exponential of $\Lambda$ once again reduces to a finite sum.   Moreover, the exponential $\exp(\zco \Lambda)$ from \eqref{finite sum comm case} has the computationally simple form $\exp(D)=\exp(r\mI_n)$.  Hence, when $i\ne j$ we have \begin{align*}
\langle \zeta_{\{i\}}\vert \zexp(\Lambda)\vert j\rangle&=\langle \zeta_{\{i\}}\vert e^r\sum_{\ell=0}^{\kappa(\Psi)}\frac{\Psi^\ell}{\ell!}\vert j\rangle\\
&=e^r\sum_{I\subseteq V}\frac{\omega_I}{|I|!} \zeta_{I},\end{align*}
where $\omega_I$ denotes the number of paths from $v_i$ to $v_j$ on vertices indexed by $I$.  Further,  when $1\le i \le n$,
\begin{equation*}
\left<i\vert \zexp(\Lambda)\vert i\right>=e^r\sum_{I\subseteq V}\frac{\omega_I}{|I|!} \zeta_{I},\end{equation*}
where $\omega_I$ denotes the number of cycles on vertex set $I$ based at $v_i\in I$.

\subsection{Diagonalization}

When $A$ is a spectrally simple square matrix of order $n$, we take as basis the eigenvectors $\epsilon=\{v_1, \ldots, v_n\}$ of $A$ (regarded as column vectors) and form the change of basis matrix $[\mI]_\epsilon^\beta=\left(v_1\mid v_2\mid \cdots\mid v_n\right)$, where $\beta$ denotes the standard ordered basis $\{e_1, \ldots, e_n\}$ of $\CZ^n$.  The resulting diagonal matrix whose entries are the zeon eigenvalues of $A$ is given by \begin{equation}\label{matrix diagonalization}
\Delta=[A]_\epsilon= [\mI]_\beta^\epsilon A[\mI]_\epsilon^\beta.
\end{equation}

It follows that the matrix exponential of $A$ is given by \begin{align}\label{diagonal exp}
\zexp(A)&=[\mI]_\epsilon^\beta\zexp(\Delta) [\mI]_\beta^\epsilon\nonumber\\
&=[\mI]_\epsilon^\beta\begin{pmatrix}
\zexp(\lambda_1)& 0&0&\cdots&0\\
0&\zexp(\lambda_2)&0&\cdots&0\\
0&0&\ddots&0 &0\\
0&0&0&\ddots&0\\
0&0&\cdots&0&\zexp(\lambda_n)
\end{pmatrix} [\mI]_\beta^\epsilon.
\end{align}

\begin{example}
Consider the spectrally simple zeon matrix \begin{equation}\label{diagonalization example}
A=\begin{pmatrix}
 \zeta _{\{1,2\}}+1 & \zeta _{\{3\}} & 2+\zeta _{\{3\}} \vspace{3pt}\\
 2 \zeta _{\{3\}} & \zeta _{\{2\}} & \zeta _{\{2\}}\vspace{3pt} \\
 2+\zeta _{\{1\}} & \zeta _{\{2\}} & 1 \\
\end{pmatrix}.
\end{equation}
The eigenvalues of $A$ are \begin{align*}
\lambda_1&= 3+\frac{1}{2} \zeta _{\{1,2\}}+\frac{1}{8} \zeta _{\{1,3\}}+\frac{1}{2} \zeta _{\{2,3\}}-\frac{1}{8} \zeta _{\{1,2,3\}}+\frac{\zeta _{\{1\}}}{2}+\frac{\zeta _{\{3\}}}{2},\\
\lambda_2&=-1+\frac{1}{2} \zeta _{\{1,2\}}-\frac{1}{8} \zeta _{\{1,3\}}+\frac{3}{2} \zeta _{\{2,3\}}-\frac{7}{8} \zeta _{\{1,2,3\}}-\frac{\zeta _{\{1\}}}{2}-\frac{\zeta _{\{3\}}}{2},\text{ \rm and}\\
\lambda_3&= -2 \zeta _{\{2,3\}}+\zeta _{\{1,2,3\}}+\zeta _{\{2\}}.
\end{align*}
We take corresponding eigenvectors to be \begin{align*}
v_1&= \left(
\begin{array}{c}
 1+\frac{1}{4} \zeta _{\{1,2\}}-\frac{1}{16} \zeta _{\{1,3\}}-\frac{1}{12} \zeta _{\{2,3\}}+\frac{17}{144} \zeta _{\{1,2,3\}}-\frac{\zeta _{\{1\}}}{4}+\frac{\zeta _{\{3\}}}{4} \vspace{3pt}\\
 -\frac{1}{18} \zeta _{\{1,2\}}-\frac{5}{18} \zeta _{\{1,3\}}+\frac{1}{6} \zeta _{\{2,3\}}-\frac{5}{72} \zeta _{\{1,2,3\}}+\frac{\zeta _{\{2\}}}{3}+\frac{2 \zeta _{\{3\}}}{3} \vspace{3pt}\\
 1 \\
\end{array}
\right),\\
v_2&= \left(
\begin{array}{c}
-1+\frac{1}{4} \zeta _{\{1,2\}}+\frac{1}{16} \zeta _{\{1,3\}}-\frac{1}{4} \zeta _{\{2,3\}}+\frac{7}{16} \zeta _{\{1,2,3\}}+\frac{\zeta _{\{1\}}}{4}-\frac{\zeta _{\{3\}}}{4}\vspace{3pt}\\
\frac{1}{2} \zeta _{\{1,2\}}-\frac{3}{2} \zeta _{\{1,3\}}-\frac{3}{2} \zeta _{\{2,3\}}+\frac{21}{8} \zeta _{\{1,2,3\}}-\zeta _{\{2\}}+2 \zeta _{\{3\}}\vspace{3pt}\\
1
\end{array}
\right),\text{ \rm and}\\
v_3&= \left(
\begin{array}{c}
\frac{4}{9} \zeta _{\{1,2\}}-\frac{2}{9} \zeta _{\{1,3\}}-\frac{4}{9} \zeta _{\{2,3\}}+\frac{13}{27} \zeta _{\{1,2,3\}}-\frac{1}{3} 2 \zeta _{\{2\}}+\frac{\zeta _{\{3\}}}{3}\vspace{3pt}\\
1\\
-\frac{2}{9} \zeta _{\{1,2\}}+\frac{1}{9} \zeta _{\{1,3\}}+\frac{2}{9} \zeta _{\{2,3\}}-\frac{11}{27} \zeta _{\{1,2,3\}}+\frac{\zeta _{\{2\}}}{3}-\frac{2 \zeta _{\{3\}}}{3}
\end{array}
\right). 
\end{align*}

Writing $[\mI]_\epsilon^\beta=(v_1\vert v_2\vert v_3)$ and $[\mI]_\beta^\epsilon=(v_1\vert v_2\vert v_3)^{-1}$, we have \begin{equation*}
[A]\epsilon=[\mI]_\beta^\epsilon A [\mI]_\epsilon^\beta=\begin{pmatrix}\lambda_1&0&0\\0&\lambda_2&0\\0&0&\lambda_3\end{pmatrix},
\end{equation*}
so that  \begin{equation*}
\zexp(A)=[\mI]_\epsilon^\beta\begin{pmatrix}
\zexp({\lambda_1})&0&0\\ 0&\zexp({\lambda_2})&0\\0&0&\zexp({\lambda_3})\end{pmatrix}[\mI]_\beta^\epsilon,
\end{equation*}
where \begin{align*}
\zexp({\lambda_1})&=e^3\left(1+\frac{1}{2} \zeta _{\{1,2\}}+\frac{3}{8}  \zeta _{\{1,3\}}+\frac{1}{2}  \zeta _{\{2,3\}}+\frac{3}{8}  \zeta _{\{1,2,3\}}+\frac{1}{2} \zeta _{\{1\}}+\frac{1}{2}  \zeta _{\{3\}}\right) ,\\
\zexp({\lambda_2})&= \frac{1}{e}\left(1+\frac{1}{2}\zeta _{\{1,2\}}+\frac{1}{8}\zeta _{\{1,3\}}+\frac{3}{2} \zeta _{\{2,3\}}-\frac{15}{8} \zeta _{\{1,2,3\}}-\frac{1}{2}\zeta _{\{1\}}-\frac{1}{2}\zeta _{\{3\}}\right),\\
\zexp({\lambda_3})&= 1-2 \zeta _{\{2,3\}}+\zeta _{\{1,2,3\}}+\zeta _{\{2\}}.
\end{align*}
\end{example}

We close with a special case of a broader result relating the determinant of the matrix exponential to the exponential of the matrix trace.

\begin{proposition}
If $A\in Mat(m;\CZ)$ is spectrally simple, then  \begin{equation}
\det(\zexp(A))=\zexp(\tr A).
\end{equation}
\end{proposition}

\begin{proof}
Letting $\{\lambda_1, \ldots, \lambda_m\}$ denote the eigenvalues of $A$ and diagonalizing $A$ as in \eqref{diagonal exp}, it follows that 

\begin{align*}
\det(\zexp(A))&=\det\left([\mI]_\epsilon^\beta \zexp(\Delta) [\mI]_\beta^\epsilon\right)\\
&=\det\left([\mI]_\epsilon^\beta\right)\det\left(\zexp(\Delta)\right)\det\left( [\mI]_\beta^\epsilon\right)\\
&=\det\left(\zexp(\Delta)\right)\\
&=\prod_{j=1}^m \zexp(\lambda_j)=\zexp\left(\sum_{j=1}^m\lambda_j\right)=\exp(\tr A).
\end{align*}
\end{proof}

\section{Conclusion and Avenues for Further Research}\label{conclusion}

Viewing a matrix $M\in\Mat(m;{\CZ})$ as a linear operator on the $\CZ$-module ${\CZ}^m$, zeon eigenvalues of $M$ have been defined as the spectrally simple zeon zeros of the characteristic polynomial of $M$.  Further, essential notions of zeon eigenvectors and linear independence have been established and used to define orthogonal projections and a resolution of the identity.
 
It has been shown that given a self-adjoint zeon matrix $A\in\Mat(m;\CZ)$  having $m$ spectrally simple eigenvalues $\lambda_1, \ldots, \lambda_m$  associated with normalized zeon eigenvectors $v_1, \ldots, v_m$, the following holds: \begin{equation*}
A=\bigoplus_{j=1}^m \lambda_j\pi_j,
\end{equation*}
where $\pi_j=v_j{v_j}^\dag$ is an orthogonal projection onto $\spn\{v_j\}$ for $j=1,\ldots, m$.

Avenues for further investigation include the following: 
\begin{enumerate}[i.]
\item a detailed treatment of zeon matrix norms and norm inequalities, which should lead to applications involving zeon matrix sequences and series;
\item properties, significance, and applications of orthogonal and unitary zeon matrices.
\end{enumerate}

\subsection*{Declarations}

\subsubsection*{Data Availability}

No data sets were generated or analyzed during the current study.

\subsubsection*{Funding}

No funding was received to assist with the preparation of this manuscript.

\subsubsection*{Conflicts of Interest/Competing Interests}

The author has no relevant financial or non-financial interests to disclose.

\end{document}